\documentclass[12pt]{amsart}

\usepackage{amsfonts}
\usepackage{amsmath}
\usepackage{amssymb}
\usepackage{amsthm}
\usepackage{enumerate}
\usepackage{geometry}
\usepackage{enumitem}
\usepackage{indentfirst}
\usepackage{etoolbox}

\DeclareMathOperator{\supp}{supp}

\DeclareMathOperator{\cl}{cl}

\DeclareMathOperator{\ordem}{o}

\theoremstyle{definition}
\newtheorem{defin}{Definition}[section]
\newtheorem{prop}[defin]{Proposition}
\newtheorem{lem}[defin]{Lemma}
\newtheorem{teo}[defin]{Theorem}
\newtheorem{cor}[defin]{Corollary}

\newtheorem{ex}[defin]{Example}

\newtheorem{quest}[defin]{Question}
\AtEndEnvironment{defin}{\null\hfill\qedsymbol}%

\numberwithin{equation}{section}

\title[Countably compact group topologies with convergent sequences]{Algebraic structure of countably compact non-torsion Abelian groups of size continuum from selective ultrafilters}

\author[M. K. Bellini]{Matheus Koveroff Bellini}
\address{Institute of Mathematics and Statistics\\ University of S\~ao Paulo\\Rua do Mat\~ao, 1010 - CEP 05508-090 - S\~ao Paulo - SP - Brazil}
\email{matheusb@ime.usp.br}

\author[A. C. Boero]{Ana Carolina Boero}
\address{Center of Mathematics, Computing and Cognition\\ Federal University of ABC\\Rua Santa Adélia, 166 - CEP 09210-170 - Santo André - SP - Brazil}
\email{ana.boero@ufabc.edu.br}

\author[V. O. Rodrigues]{Vinicius de Oliveira Rodrigues}
\address{Institute of Mathematics and Statistics\\ University of S\~ao Paulo\\Rua do Mat\~ao, 1010 - CEP 05508-090 - S\~ao Paulo - SP - Brazil}
\email{vinior@ime.usp.br}

\author[A. H. Tomita]{Artur Hideyuki Tomita}
\address{Institute of Mathematics and Statistics\\ University of S\~ao Paulo\\Rua do Mat\~ao, 1010 - CEP 05508-090 - S\~ao Paulo - SP - Brazil}
\email{tomita@ime.usp.br}

\thanks{MSC: primary 54H11, 22A05; secondary 54A35, 54G20}

\keywords{countable compactness, convergent sequences, topological group}

\begin{document}

\maketitle

\begin{abstract}
Assuming the existence of $\mathfrak c$ incomparable selective ultrafilters, we classify the non-torsion Abelian groups of cardinality $\mathfrak c$ that admit a countably compact group topology. We show that for each $\kappa \in [\mathfrak c, 2^\mathfrak c]$ each of these groups has a countably compact group topology of weight $\kappa$ without non-trivial convergent sequences and another that has convergent sequences. 

Assuming the existence of $2^\mathfrak c$ selective ultrafilters, there are at least $2^\mathfrak c$ non homeomorphic such topologies in each case and we also show that every Abelian group of cardinality at most $2^\mathfrak c$ is algebraically countably compact. We also show that it is consistent that every  Abelian group of cardinality $\mathfrak c$ that admits a countably compact group topology admits a countably compact group topology without non-trivial convergent sequences whose weight has countable cofinality. 
\end{abstract}

\section{Introduction}

\subsection{Some history}

Under Martin's Axiom, Dikranjan and Tkachenko \cite{dikranjan&tkachenko} showed that if $G$ is a non-torsion Abelian group of size continuum, then the following conditions are equivalent:

\begin{enumerate}[label=\alph*)]
  \item $G$ admits a countably compact Hausdorff group topology without non-trivial convergent sequences;
  \item $G$ admits a countably compact Hausdorff group topology;
  \item the free rank of $G$ is equal to $\mathfrak{c}$ and, for all $d, n \in \mathbb{N}$ with $d \mid n$, the group $dG[n]$ is either finite or has cardinality $\mathfrak{c}$.

\end{enumerate}

The implications a)$\rightarrow$b) and b)$\rightarrow$c) hold in ZFC \cite{dikranjan&tkachenko}.\\

The classification of all torsion groups of arbitrary cardinality using a single selective ultrafilter and a mild cardinal arithmetic hypothesis appears in Castro-Pereira and Tomita \cite{castro-pereira&tomita2010}. The example for torsion groups can be $p$-compact for the selective ultrafilter $p$ and this cannot be expected for non-torsion Abelian groups. Indeed, there are no $p$-compact group topologies on free Abelian groups \cite{tomita2}.

Boero and Tomita \cite{boero&tomita1} showed that the almost torsion-free groups of cardinality $\mathfrak c$ admit a countably compact group topology using $\mathfrak c$ selective ultrafilters. 

In \cite{boero&garcia-ferreira&tomita}, it was shown that the free Abelian group of cardinality $\mathfrak c$ admits a countably compact group with a convergent sequence using $\mathfrak c$ selective ultrafilters. In 
\cite{bellini&boero&castro&rodrigues&tomita} it was shown from $\mathfrak p = \mathfrak c$ that every Abelian group of cardinality $\mathfrak c$ that admits a countably compact group topology also admits a countably compact group topology with a convergent sequence.

Malykhin and Shapiro  \cite{malykhin&shapiro} showed that every pseudocompact Abelian group whose weight has countable cofinality has a convergent sequence and showed a forcing example of a pseudocompact group topology without non-trivial convergent sequences whose weight is $\aleph_1 < (\aleph_1)^\omega$ for the group $(\mathbb Z_2)^{(\mathfrak c)}$. 
Tomita \cite{tomita4} showed that it is consistent that  $\mathbb Z^{(\mathfrak c)}$ admits a countably compact group topology without non-trivial convergent sequences whose weight is $\aleph_\omega$. 

In this paper, we improve the main results of both \cite{bellini&boero&castro&rodrigues&tomita} and \cite{dikranjan&shakhmatov2}. Indeed, we show from the existence of $\mathfrak c$ selective ultrafilters that condition c) implies condition  a), d) and e) where d) and e) are as follows:\\

\begin{enumerate}[label=\alph*)]
\setcounter{enumi}{3}
\item $G$ admits a countably compact Hausdorff group topology with a convergent sequence.\\

\item for each $\kappa \in [\mathfrak c,2^\mathfrak c]$, $G$ admits a countably compact Hausdorff group topology without convergent sequences whose weight is $\kappa$.\\
\end{enumerate}

Assuming $2^\mathfrak c$ selective ultrafilters, there are $2^\mathfrak c$ non-homeomorphic such group topologies satisfying a), b), d) and e). We note that Martin's Axiom implies the existence of $2^\mathfrak c$ selective ultrafilters; on the other hand the existence of $2^ \mathfrak c$ selective ultrafilters is consistent with the negation of Martin's Axiom.

Also using $2^\mathfrak c$ selective ultrafilters, we show that every Abelian group of cardinality at most $2^\mathfrak c$ is algebraically countably compact.

Recently Hru\v sák, van Mill, Ramos-Garcia and Shelah proved there exists a countably compact group without non-trivial convergent sequences in ZFC \cite{michaelnew}. It is not known yet if their technique can be used to produce a countably compact free Abelian group without non-trivial convergent sequences in ZFC. This would be a first step to classify in ZFC the non-torsion Abelian groups of cardinality $\mathfrak c$ that admit a countably compact group topology.

Also recently, Bellini, Rodrigues and Tomita  proved that $\oplus_{\xi < \mathfrak c} \mathbb Q$ admits a $p$-compact group topology (a $p$-compact space is a space such that every sequence has a $p$-limit) \cite{BELLINI2021107653}. It is natural to ask for the classification of all groups of cardinality $\mathfrak c$ that admits a $p$-compact group topology. Such classification would be different from the one here since it is known that the free Abelian group does not admit a $p$-compact group topology,

\subsection{Basic results, notation and terminology}

Recall that a $T_1$ topological space is \emph{countably compact} iff every infinite subspace has an accumulation point in the space.

The following definition was introduced in \cite{bernstein} and is closely related to the definition of accumulation point.

\begin{defin}\label{def_p-limit}
Let $p$ be a free ultrafilter on $\omega$ and let $s:\omega\rightarrow X$ be a sequence in a topological space $X$. We say that $x \in X$ is a \emph{$p$-limit point} of $s$ if, for every neighborhood $U$ of $s$, $\{n \in \omega : s(n) \in U\} \in p$.
\end{defin}

 If $X$ is a Hausdorff space a sequence $s$ has at most one $p$-limit point $x$ and we write $x = p$-$\lim s$.

The set of all free ultrafilters on $\omega$ will be denoted by $\omega^{*}$. It is not difficult to show that a $T_1$ topological space $X$ is countably compact iff each sequence in $X$ there exists $p \in \omega^*$ such that $s$ has a $p$-limit point.

The following proposition, whose proof is well known and straightforward, tell us how the $p$-limits behave with respect to products. 
\begin{prop}\label{prop_p-limit_product}
If $p \in \omega^{*}$ and $(X_i : i \in I)$ is a family of topological spaces, then $(y_{i})_{i \in I} \in \prod_{i \in I} X_i$ is a $p$-limit point of a sequence $((x_{i}^{n})_{i \in I} : n \in \omega)$ in $\prod_{i \in I} X_i$ if, and only if, $y_i$ is a 
$p$-limit point of $(x_{i}^{n} : n \in \omega)$ for every $i \in I$.\qed
\end{prop}

The following proposition is also very easy to prove. A proof can be found in \cite{hindman2011algebra} (Theorem 3.54).

\begin{prop}
	Let $X$, $Y$ be topological spaces and $f:X\rightarrow Y$ be a continuous function, $s:\omega\rightarrow X$ be a sequence in $X$ and $p \in \omega^*$. It follows that if $x=p$-$\lim (s_n: n \in \omega)$, then $f(x)=p$-$\lim (f(s_n): n \in \omega)$.\qed
\end{prop}
Since $+$ and $-$ are continuous functions in topological groups, it follows from the two previous propositions that:
\begin{prop}\label{prop_p-limit}
Let $G$ be a topological group and $p \in \omega^{*}$.

\begin{enumerate}

  \item If $(x_n : n \in \omega)$ and $(y_n : n \in \omega)$ are sequences in $G$ and $x, y \in G$ are such that $x = p$-$\lim(x_n : n \in \omega)$ and $y = p$-$\lim(y_n : n \in \omega)$, then $x + y = p$-$\lim (x_n + y_n : n \in \omega)$;

  \item If $(x_n : n \in \omega)$ is a sequence in $G$ and $x \in G$ is such that $x = p$-$\lim(x_n : n \in \omega)$, then $- x = p$-$\lim(- x_n : n \in \omega)$.\qed

\end{enumerate}
\end{prop}

The unit circle group $\mathbb{T}$ will be identified with the metric group $(\mathbb{R} / \mathbb{Z}, \delta)$ where $\delta$ is given by $\delta(x + \mathbb{Z}, y + \mathbb{Z}) = \min\{|x - y + a| : a \in \mathbb{Z}\}$, for every $x, y \in \mathbb{R}$. Given a subset $A$ of $\mathbb{T}$, we will denote by $\delta(A)$ the diameter of $A$ with respect to the metric $\delta$. The set of all non-empty open arcs of $\mathbb{T}$ will be denoted by $\mathcal{B}$.

We also fix $\mathcal U$ a basis of $\mathbb T^\mathfrak c$ of cardinality $\mathfrak c$ consisting of basic open sets which are products of elements of $\mathcal B$. As such, for each $U\in\mathcal U$ and $\mu\in\mathfrak c$ we define $U_\mu\in\mathcal B$ to be the $\mu$-th coordinate of $U$ and let $\supp U=\{\mu\in\mathfrak c:U_\mu\ne\mathbb T\}$.

Let $X$ be a set and $G$ be a group. We denote by $G^{X}$ the set of all functions from $X$ into $G$. The \emph{support} of $g \in G^{X}$ is the set $\{x \in X : g(x) \neq 0\}$, which will be designated as $\supp g$. The set $\{g \in G^{X} : |\supp g| < \omega\}$ will be denoted by $G^{(X)}$. If $f:\omega\rightarrow G^{(X)}$ is a sequence, then $\supp f=\bigcup_{n \in \omega} \supp  f(n)$.

The torsion part $T(G)$ of an Abelian group $G$ is the set $\{x \in G : \exists n \in \mathbb N ( nx = 0) \}$. Clearly, $T(G)$ is a subgroup of $G$. For every $n \in \mathbb{N}$, we put $G[n] = \{x \in G : nx = 0\}$. Given a natural number $n>0$, we say that $G$ is \emph{of exponent $n$} provided $G = G[n]$ and that $n$ is the minimal positive integer with this property.

The order of an element $x \in G$ will be denoted by $\ordem(x)$.

A non-empty subset $S$ of an Abelian group $G$ is said to be \emph{independent} if $0 \not \in S$ and, given distinct elements $s_1, \ldots, s_n$ of $S$ and integers $m_1, \ldots, m_n$, the relation $m_1 s_1 + \ldots + m_n s_n = 0$ implies that $m_i s_i = 0$, for all $i$. The free rank $r(G)$ of $G$ is the cardinality of a maximal independent subset of $G$ such that all of its elements have infinite order. It is easy to verify that $r(G) = |G / T(G)|$ if $r(G)$ is infinite.

An Abelian group $G$ is called \emph{divisible} if, for each $g \in G$ and each $n \in \mathbb{N} \setminus \{0\}$, the equation $nx = g$ has a solution $x \in G$.

\subsection{Notation for some specific groups} In this section we fix some of the notation that will be used throughout this article.

We fix a two element partition $\{P_0, P_1\}$ of $\mathfrak c$ such that $|P_0|=|P_1|=\mathfrak c$ and  $\omega+1\subseteq P_1$.

\begin{defin} We define the groups $\mathbb W^0=(\mathbb Q/\mathbb Z)^{(P_0\times\omega)}$, $\mathbb W^1=\mathbb Q^{(P_1)}$, $\mathbb W=\mathbb W^0\oplus \mathbb W^1$, $\mathbb X^0=\mathbb Q^{(P_0\times\omega)}$, $\mathbb X^1=\mathbb Q^{(P_1)}$, $\mathbb X=\mathbb X^0\oplus \mathbb X^1$.

Given $w \in \mathbb{W}$ (or $w \in {\mathbb X}$), we denote by $w^0$ and $w^1$ the respective projections to $\mathbb{W}^0$ and $\mathbb{W}^1$
(${\mathbb X}^0$ and ${\mathbb X}^1$); that is,  the unique elements of $\mathbb{W}$ (or $\mathbb X$) such that $\supp w^0 \subseteq {P_0}\times\omega$, $\supp w^1 \subseteq P_1$ and $w=w^0+w^1$.

Similarly, given $g$ a sequence from $\omega$ into $ \mathbb{W}$ (or ${\mathbb X}$), we define $g^0$ and $g^{1}$. So $g=g^0+g^1$, $\supp g^0\subseteq {P_0}\times\omega$ and $\supp g^1\subseteq P_1$.
\end{defin}
It will be useful to be able to easily transform an element of $\mathbb X$ into an element of $\mathbb{W}$, so we use the following definition: 
\begin{defin}
 We denote by $[\,\cdot\,]$ the homomorphism from ${\mathbb X}={\mathbb X}^0\oplus {\mathbb X}^1$ onto $\mathbb{W} = \mathbb{W} ^0 \oplus \mathbb{W}^1$ so that ${\mathbb X}^0 \to {\mathbb{W}}^0$ is the quotient map coordinatewise and ${\mathbb X}^1 \to {\mathbb{W}}^1 = {\mathbb X}^1$ is the identity.
 
  Given a function $g:\omega\rightarrow \mathbb X$, we also define $[g]:\omega\rightarrow \mathbb W$ be given by $[g](n)=[g(n)]$ for every $n \in \omega$.
\end{defin}

Now we introduce the notion of a \textit{nice subgroup} of $\mathbb W$. These will be the main groups we will be working with throughout the article. Associated to each nice subgroup $\textbf{W}$ of $\mathbb W$ we have a subgroup $\mathbf Q$ of $\mathbb X$ of the elements which generate the elements of $\mathbf W$ and the group $\mathbf Z$ of elements of integer coordinates of $\mathbf Q$. All these groups will be useful in the construction.

\begin{defin}Given a family of positive integers $\vec n=(n_\xi: \xi \in P_0)$, we define $\vec  P_{0, \vec n}=\bigcup_{\xi \in P_0}\{\xi\}\times n_\xi$.

We say that $\mathbf{W}$ is a nice subgroup of $\mathbb W$ if there exists a (unique) family of positive integers $\vec n=(n_\xi: \xi \in P_0)$ such that $\mathbf W=(\mathbb Q/\mathbb Z)^{(\vec P_{0, \vec n})}\oplus \mathbb Q^{(P_1)}$.

In this case, we define $\mathbf{W}^0=(\mathbb Q/\mathbb Z)^{(\vec P_{0, \vec n})}$ and $\mathbf{W}^1=\mathbb Q^{(P_1)}$. Moreover, we define the following subgroups of $\mathbb X$: $\mathbf Q_{\mathbf W}^0=\mathbb Q^{(\vec P_{0, \vec n})}$, $\mathbf Q_{\mathbf W}^1= \mathbb Q^{(P_1)}$ and $\mathbf Q_{\mathbf W}= \mathbf Q_{\mathbf W}^0 \oplus \mathbf Q_{\mathbf W}^1$,
$\mathbf Z_{\mathbf W}^0=\mathbb Z^{(\vec P_{0, \vec n})}$, $\mathbf Z_{\mathbf W}^1=\mathbb Z^{(P_1)}$ and ${\mathbf Z}_\mathbf W=\mathbf Z_{\mathbf W}^0\oplus \mathbf Z_{\mathbf W}^1$.
\end{defin}

Note that, given a nice subgroup $\mathbf W$, the restriction of $[\,\cdot\,]$ to $\mathbf Q_\mathbf W$ is a homomorphism onto $\mathbf W$.

Throughout this article we will never deal with more than one nice subgroup of $\mathbb W$ at the same time. Consequently, since no confusion arrises, given a nice subgroup $\mathbf W$ of $\mathbb W$ we may just write $\vec P_0=\vec P_{0, \vec n}$, $\mathbf Q=\mathbf Q_{\mathbf W}$, $\mathbf Z=\mathbf Z_{\mathbf W}$, and $\mathbf Q^j=\mathbf Q_{\mathbf W}^j$, $\mathbf Z^j=\mathbf Z_{\mathbf W}^j$ for $j=0$ or $1$.

The following notation is also useful:

\begin{defin}
Given $E\subseteq \mathfrak c$, we define $E^0=E\cap P_0$ and $E^1=E\cap P_1$.
\end{defin}

\begin{defin} Given $\xi \in P_1$, the function $\chi_{\xi} : P_1 \to \mathbb{Q}$ is given by \[\chi_{\xi}(\mu) = \left\{ \begin{array}{lll}
                                                      1 & \hbox{if} & \mu = \xi \\
                                                      0 & \hbox{if} & \mu \neq \xi.
                                                    \end{array}
 \right.\]
\end{defin}

\begin{defin}if $r \in  \mathbb Q/\mathbb Z$, then $p(r)$ and $q(r)$ are the unique integers $p, q$ such that $q>0$, $\gcd(p, q)=1$, $0\leq p<q$ and $r=\frac{p}{q}+\mathbb Z$. Likewise, if $r \in \mathbb Q$,  $p(r)$ and $q(r)$ are the unique integers $p, q$ such that $q>0$, $\gcd(p, q)=1$ and $r=\frac{p}{q}$.
 
 We also define, for $w\in\mathbf W$, $p(w) = \max\{|p(w(z))| : z \in \supp w\}$  and $q(w) = \max\{q(w(z)) : z \in \supp w \}$ if $w \neq 0$. We define $p(0)=0$ and $q(0)=0$.
\end{defin}

 \subsection{Structure of the article}
 
In Section \ref{section.selective.ultrafilters}, we recall the definition of selective ultrafilters and some properties which are necessary to construct homomorphisms restricted to countable subgroups.

In Section 
\ref{section.arc}, we recall the notion of arc homomorphism that was defined in \cite{bellini&boero&castro&rodrigues&tomita} which helped simplify the construction of 
the homomorphisms.

In Section \ref{section.types}, we define the types and state the theorem in \cite{bellini&boero&castro&rodrigues&tomita} that shows that every 1-1 sequence can be associated to one of the types.

In Section \ref{section.arc.hom}, we improve the result in \cite{bellini&boero&castro&rodrigues&tomita} to have an estimate of the size of the output arc that depends only on the size of the input arc.

In Section \ref{section.countable.hom}, we prepare the arc homomorphism to make the sequences have the pre-assigned $p$-limits and we produce the homomorphisms on countable subgroups. Here lie the main differences between constructing homomorphisms using $\mathfrak p=\mathfrak c$ and selective ultrafilters. For selective ultrafilters, we have to know the size of the output arcs in advance to guarantee the assigned points will be the $p$-limits.

In Section \ref{section.proof.main}, we present the algebraic immersion, extend the homomorphism on countable subgroups and define group topologies with a non-trivial  convergent sequence and without non-trivial convergent sequences.

In Section \ref{moreexamples} we show that the main construction can be used to obtain some more examples which were already mentioned in the introduction.

\section{Selective Ultrafilters} \label{section.selective.ultrafilters}
	In this section we review some basic facts about selective ultrafilters, the Rudin-Keisler order and some lemmas we will use in the next sections.
	\begin{defin}A selective ultrafilter (on $\omega$), also called Ramsey ultrafilter, is a free ultrafilter $p$ on $\omega$ such that for every partition $(A_n: n \in \omega)$ of $\omega$, either there exists $n$ such that $A_n \in p$ or there exists $B \in p$ such that $|B\cap A_n|=1$ for every $n \in \omega$. 
	\end{defin}
	
	The following proposition is well known. We provide \cite{jech} as a reference.
	\begin{prop}Let $p$ be a free ultrafilter on $\omega$. Then the following are equivalent:
		\begin{enumerate}[label=\alph*)]
			\item $p$ is a selective ultrafilter,
			\item for every $f \in \omega^\omega$, there exists $A \in p$ such that $f|_A$ is either constant or one-to-one,
			\item for every function $f:[\omega]^2\rightarrow 2$ there exists $A \in p$ such that $f|_{[A]^2}$ is constant.
		\end{enumerate}
	\end{prop}
	
	The Rudin-Keisler order is defined as follows:
	\begin{defin}
		
		Let $\mathcal U$ be a filter on $\omega$ and $f:\omega\rightarrow \omega$. We define $f_*(\mathcal U)=\{A \subseteq \omega: f^{-1}[A] \in \mathcal U\}$.
	\end{defin}
	
	It is easy to verify that: $f_*(\mathcal U)$ is a filter; if $\mathcal U$ is an ultrafilter, so is $f_*(\mathcal U)$; if $f, g:\omega\rightarrow \omega$, then $(f\circ g)_*=f_*\circ g_*$; and $(\text{id}_\omega)_*$ is the identity over the set of all filters. This implies that if $f$ is bijective, then $(f^{-1})_*=(f_*)^{-1}$.
	\begin{defin}
		
		Let $p$, $q$ be ultrafilters on $\omega$. We say that $p\leq q$ (or $p\leq_{\text{RK}} q$), if we need to avoid ambiguity) iff there exists $f \in \omega$ such that $f_{*}(p)=q$.
		
		The \textit{Rudin-Keisler order} is the set of all free ultrafilters over $\omega$ ordered by $\leq_{\text{RK}}$. We say that two ultrafilters $p, q$ are equivalent iff $p\leq q$ and $q\leq p$.
	\end{defin}
	
	It is easy to verify that $\leq$ is a preorder and that the equivalence defined above is in fact  an equivalence relation. Moreover, the equivalence class of a fixed ultrafilter is the set of all fixed ultrafilters, so the relation restricts to $\omega^*$ without modifying the equivalence classes. We refer to \cite{jech} for the following proposition:
	
	\begin{prop}The following are true:
		\begin{enumerate}
			\item If $p, q$ are ultrafilters, then $p\leq q$ and $q\leq p$ is equivalent to the existence of a bijection $f:\omega\rightarrow \omega$ such that $f_*(p)=q$.
			\item The selective ultrafilters are exactly the minimal elements of the Rudin-Keisler order.
		\end{enumerate}
		
	\end{prop}
	
	This implies that if $f:\omega\rightarrow \omega$ and $p$ is a selective ultrafilter, then $f_*(p)$ is either a fixed ultrafilter or a selective ultrafilter. If $f_*(p)$ is the ultrafilter generated by $n$, then $f^{-1}[\{n\}] \in p$, so, in particular, if $f$ is finite to one and $p$ is selective, then $f_*(p)$ is a selective ultrafilter equivalent to $p$.
	
	The existence of selective ultrafilters is independent from ZFC. Martin's Axiom for countable orders implies the existence of $2^{\mathfrak c}$ pairwise incompatible selective ultrafilters in the Rudin-Keisler order.

	The lemma below appears in \cite{tomita3} (Lemma 3.6) and will be used for the construction of homomorphisms on countable subgroups.

	\begin{lem}\label{countable.ult}
		
		Let $({p}_k:\, k \in \omega)$ be a family of pairwise incomparable selective ultrafilters. For each $k$, let $(a_{k, i}: i \in \omega)$ be a strictly increasing sequence in $\omega$ such that $\{a_{k,i}:\, i \in \omega\} \in {p}_k$ and $i<a_{k, i}$ for each $k, i \in \omega$. Then there exists $\{I_k:\, k \in \omega \}\subset[\omega]^\omega$ such that:
		
		\begin{enumerate}[label=\alph*)]
			\item $\{ a_{k,i}:\, i \in I_k\} \in {p}_k$, for each $k \in \omega$.
			
			\item $I_i\cap I_j=\emptyset$ whenever $i, j \in \omega$ and $i\neq j$, and 
			
			\item $([i,a_{k,i}]:\, k \in \omega \text{ and } i \in I_k)$ is a pairwise disjoint family.
		\end{enumerate}
		
	\end{lem}
	
\section{Arc homomorphisms} \label{section.arc}

Throughout this section we fix a nice subgroup $\mathbf W$ of $\mathbb W$. 

In \cite{bellini&boero&castro&rodrigues&tomita}, we introduced the concept of arc homomorphism to construct homomorphisms from countable subgroups of $\mathbf W$ to $\mathbb T$.
This technique replaces the arc functions that were used in \cite{tomita2015} to treat free Abelian groups.

In this section, we recall this concept and some of the useful results. More details can be seen in \cite{bellini&boero&castro&rodrigues&tomita}.

First, we define some subgroups of $\mathbf{W}$. A countable subset of $\mathfrak c$ generates a countable subset of $\mathbf{W}$. We may restrict this subgroup to the elements which are annihilated by an integer $K$. Formally:

\begin{defin}Given  $E$ a subset of $\mathfrak c$, denote by 

$$\mathbf{W}_{E} = \{ w \in \mathbf{W}:\, \supp w \subseteq (E^0 \times \omega) \cup E^1\}.$$

Given positive integer $K$, we define $$\mathbf{W}_{(K)} =\{ w \in \mathbf{W}:\,  \exists u \in {\mathbf Q}, w=[u] \text{ and } Ku \in \mathbf Z \}.$$ 

Finally denote $$\mathbf{W}_{E,K}= \mathbf{W}_E \cap \mathbf{W}_{(K)}.$$
\end{defin}

In particular, if $E\subseteq P_0$, then $\mathbf{W}_{E, K}$ is the subgroup of ${(\mathbb Q/\mathbb Z)}^{(\vec{P_0} \cap (E\times \omega))}$ of the elements which are annihilated by $K$, and if $E\subseteq P_1$, then $\mathbf{W}_{E, K}$ is the subgroup of $\mathbb Q^{(E)}$ of the elements that become an element of $\mathbb Z^{(E)}$ when multiplied by $K$. Also, notice that $\mathbf{W}_{E, K}=\mathbf{W}_{E^0, K}\oplus\mathbf{W}_{E^1, K}$ is a finitely generated group whenever $E$ is finite.

Now we recall the definition of an $(E,K,\epsilon)$-arc homomorphism.

\begin{defin}Given a positive real $\epsilon$, a subset $E$ of $\mathfrak c$ and a positive integer $K$, an $(E, K, \epsilon)$-arc homomorphism is a pair $\phi=(\phi^0, \phi^1)$ where:

\begin{enumerate}[label=\alph*)]
\item $\phi^0: \mathbf{W}_{E^0, K}\to \mathbb T$ is a homomorphism,
\item $\phi^1: \{\frac{1}{K}\chi_\xi:\, \xi \in E^1 \}\rightarrow \mathcal B$,
\item $\phi^1(\frac{1}{K}\chi_\xi)$ is an arc of length $\epsilon$, for every $\xi \in E^1$.\end{enumerate}\end{defin}

$\phi^1$ can naturally be extended to $\mathbf{W} _{E^1,K}$ as follows:

\begin{defin}Let $E\subseteq \mathfrak c$, $K$ be a positive integer, $\epsilon>0$, $\phi=(\phi^0, \phi^1)$ be an $(E, K, \epsilon)$-arc homomorphism and $w \in \mathbf{W}_{E,K}$. Then we define:

$$\hat\phi(w)=\phi^0(w^0)+\sum_{\xi \in \supp w^1}Kw(\xi)\phi^1\left(\frac{\chi_\xi}K\right)$$

We define the empty sum of arcs to be the set $\{0+\mathbb Z\}$.
\end{defin}

Notice that if $\supp w \subseteq \vec{P_0}$ then $\hat \phi(w)=\{\phi^0(w)\}$.

As in \cite{bellini&boero&castro&rodrigues&tomita}, the following notion of extension will be very useful:

\begin{defin}
Given $\epsilon', \epsilon>0$, positive integers $K, K'$, subsets $E, E'$ of $\mathfrak c$, $\phi=(\phi^0, \phi^1)$ an $(E, K, \epsilon)$-arc homomorphism and $\psi=(\psi^0, \psi^1)$ an $(E', K', \epsilon')$-arc homomorphism, we say that $\psi < \phi$ ($\psi$ extends $\phi$) iff:

\begin{enumerate}[label=\alph*)]
\item $E\subseteq E'$, $K|K'$ and $\epsilon'\leq \epsilon$,
\item $\phi^0\subseteq \psi^0$,

\item  $\cl\left(\frac{K'}{K}\psi^1(\frac{1}{K'}\chi_\xi)\right) \subseteq \phi^1(\frac{1}{K}\chi_\xi)$, for each $\xi \in E^1$.

\item $m!\frac{K'}{K}\psi^1(\frac{\chi_m}{K'})$ is contained in the arc $(-\frac\epsilon {2K}, \frac \epsilon {2K})+\mathbb Z$, for each $m \in (E' \cap \omega) \setminus E$.
\end{enumerate}
\end{defin}

The following proposition appears in \cite{bellini&boero&castro&rodrigues&tomita} and its  proof is straightforward.

\begin{prop}\label{homomorphismconstructor}Suppose $(\phi_n: n \in \omega)$ is a sequence of $(E_n, K_n, \epsilon_n)$-arc homomorphisms such that: $\phi_{n+1}<\phi_n$ for every $n \in \omega$; $(\epsilon_n: n \in \omega)\rightarrow 0$; and each positive integer divides cofinitely many elements of $(K_n: n \in \omega)$. Let $E=\bigcup_{n \in \omega} E_n$. Then:

\begin{enumerate}
\item For every $w \in \mathbf{W}_E$, $\bigcap\{\hat \phi_n(w): w \in \mathbf{W}_{E_n, K_n}\}$ is a singleton.
\item The function $\psi: \mathbf{W}_E\to \mathbb T$ so that $\psi(w) \in \bigcap\{\hat \phi_n(w): w \in \mathbf{W}_{E_n, K_n}\}$ is a group homomorphism.
\item If $\omega\subseteq E$ and $E_n$ is finite for every $n$, then for every positive integer $S$, $(\psi\left(\frac{m!}{S}\chi_m\right): m \in \omega)$ converges to $0\in \mathbb T$.
\end{enumerate} \qed
\end{prop}

\section{Types of sequences} \label{section.types}

\subsection{Introduction}

Throughout this section we fix a nice subgroup $\mathbf W$ of $\mathbb W$. We will recall, from  \cite{bellini&boero&castro&rodrigues&tomita}, the definition of the 11 types of sequences in the next subsection and the theorem that states that every sequence is related to one of them. The types are defined with respect to a subgroup $G$ of $\mathbf{W}$.

We call $\mathcal H$ the class of all the sequences that are of one of the 11 types.

The theorem that we are going to need, proved in \cite{bellini&boero&castro&rodrigues&tomita}, is the following:
\begin{teo}\label{prop_subseq}
Let $f : \omega \to \mathbf Q$ be given by $f(n) = n!   \chi_n$ for every $n \in \omega$. Let $G$ be a subgroup of $\mathbf{W}$ containing $\chi_n$ for every $n \in \omega$. Let $g: \omega \to \mathbf Q$ with $[g] \in G^\omega$. 

Then there exists $h:\omega\rightarrow \mathbf Q$ such that $h\in \mathcal{H}$ or $[h]$ is a constant sequence in $G^\omega$, $c\in \mathbf Q$ with $[c] \in G$, $F \in [\omega]^{< \omega}$, $p_i, q_i \in \mathbb{Z}$ with $q_i \neq 0$ for every $i \in F$, $(j_{i}:i\in F)$ increasing enumerations of subsets of $\omega$  and $j: \omega \to \omega$ strictly increasing such that

\[g \circ j = h + c + \sum_{i \in F} \frac{p_i}{q_i}   f\circ j_{i}\]

with $q_i\leq j_i(n)$ for each $n \in \omega$ and $i \in F$ (which implies $[\frac{p_i}{q_i} f\circ j_i] \in G^\omega$ since $q_i|((j_i(n))!)$ for each $i \in F$ and $n \in \omega$).\qed
\end{teo}

A discussion on the types can be seen in \cite{bellini&boero&castro&rodrigues&tomita}.

To state the definition of the 11 types, the following notation is useful:

Given $w \in \mathbf Q$, we call $w^{1,0}$ and $w^{1,1}$ the natural projections of $w$ into $\mathbb Q^{(\omega)}$ and $\mathbb Q^{(P_1\setminus \omega)}$.

Similarly, given $g:\omega\to \mathbf{W}$ (or $\mathbf Q$), we define $g^{1, 0}$ and $g^{1, 1}$ so that  $g^1=g^{1, 0}+g^{1, 1}$, with $ g^{1, 0}\in ({\mathbf Q}_ \omega)^\omega$ and $ g^{1, 1}\in ({\mathbf Q}_{P_1 \setminus \omega})^\omega$.

\subsection{The types} The types are grouped by where the main element of the support is.

\subsubsection{The types related to $P_1 \setminus \omega$}

\begin{defin}Let $G$ be a subgroup of $\mathbf{W}$. We define the first three types of sequences (with respect to $G$) as follows:\label{def_tipos_P1}
Let $g : \omega \to \mathbf Q$ be such that $[g]\in G^\omega$. \\

We say that $g$ is \emph{of type 1} if $\supp g^{1,1}(n) \setminus \bigcup_{m < n} \supp g^{1,1}(m) \neq \emptyset$, for every $n \in \omega$.\\

We say that $g$ is \emph{of type 2} if $q(g^{1,1}(n)) > n$, for every $n \in \omega$.\\

We say that $g$ is \emph{of type 3} if $\{q(g^{1,1}(n)) : n \in \omega\}$ is bounded and $|p(g^{1,1}(n))| > n$, for every $n \in \omega$.
\end{defin}

\subsubsection{The types related to $\omega$}

\begin{defin}
Let $G\subseteq \mathbf{W}$ be a subgroup. Let $g : \omega \to \mathbf Q$ be such that $[g] \in G^\omega$. Then we define types 4 to 9 (with respect to $G$) as follows: \\

We say that $g$ is \emph{of type 4} if $q(g(n)) > n$, for every $n \in \omega$.\\

We say that $g$ is \emph{of type 5} if there exists $M \in \bigcap_{n \in \omega} \supp g^{1,0}(n)$ such that $\{q(g(n)) : n \in \omega\}$ is bounded and $|p(g(n)(M))| > n$, for every $n \in \omega$.\\

To define types 6, 7 and $8$, suppose $g$ is  such that for each $n \in \omega$, there exists $M_n \in \supp g^{1,0}(n) \setminus \cup_{m < n} \supp g^{1,0}(m)$ such that \[\left( \frac{g(n)(M_n)}{M_n !} : n \in \omega \right)\]

\noindent
is a 1-1 sequence that converges to some $ u \in  (\mathbb R\setminus \mathbb Q)\cup \{-\infty, 0, \infty\}.$
\\

We say that $g$ is of \emph{type 6} if $u=0$.\\

We say that $g$ is of \emph{type 7} if $u \in \mathbb R\setminus \mathbb Q$.\\

We say that $g$ is of \emph{type 8} if $u$ is $\infty$ or $-\infty$. \\

We say that $g$ is \emph{of type 9} if $\left\{ \dfrac{g(n)(M)}{M!} : M \in \supp g^{1,0}(n), n \in \omega \right\}$ is finite and $|\supp g^{1,0}(n)| > n$ for every $n \in \omega$.
\end{defin}

\subsubsection{The types related to $\vec{P_0}$}

\begin{defin}
We define types 10 and 11 (with respect to $G$) as follows:
Let $g : \omega \to \mathbf Q$ be such that $[ g]\in (\mathbf{W}^0)^\omega \cap   G^\omega$. \\

We say that $g$ is \emph{of type 10} if $q(g^0(n)) > n$, for every $n \in \omega$.\\

We say that $g$ is \emph{of type 11} of order $k$ if  the family $\{[g(n)]: n \in \omega\}$ is an independent family whose elements have a fixed order $k$, for some positive integer $k$.
\end{defin}

 \section{Extensions of Arc homomorphisms } \label{section.arc.hom}
 
 Throughout this section we fix a nice subgroup $\mathbf W$ of $\mathbb W$.
 
We will state theorems on extensions of arc homomorphisms related to the types which will be important to build homomorphisms in countable subgroups of $\mathbf W$. We have to make an estimate for the sizes of the output arcs compared to the input arcs. 

In \cite{bellini&boero&castro&rodrigues&tomita}, the step for the construction of the arc homomorphisms was divided in two theorems. Theorem 4.1 was used to treat sequences of types 1 to 10, and Theorem 4.2 was used to treas sequences of type 11.
Here, we will rewrite those theorems from \cite{bellini&boero&castro&rodrigues&tomita} in a unique statement. In fact, the roles of $\epsilon'$, $K'$, $k$ and $\epsilon^*$ as stated below is a bit different than the roles they have in Theorems 4.1 and 4.2 of  \cite{bellini&boero&castro&rodrigues&tomita}, however, the reader may verify that the version of the theorem as stated below is a trivial consequence of those previous theorems.

\begin{teo} \label{arc_extension_new}Let the following be given:
\begin{enumerate}[label=\roman*)]
    \item $E\in [\mathfrak c]^{<\omega}$,
    \item $K$ a positive integer,
    \item $\epsilon^*,\epsilon$ positive reals with $\epsilon<\frac{1}{2}$,
    \item $\phi$ an $(E, K, \epsilon)$-arc homomorphism,
    \item $h$ a sequence of type 1 to 11,
    \item $\gamma>0$ if $h$ is of type 1 to 10,
    \item $k$ as the order of $h$, if $h$ is of type 11, or $k=1$, if $h$ is not of type 11.
\end{enumerate}

Then there exists a cofinite set $A \subseteq \omega$ such that for every $n \in A$, for every finite $E'\supseteq E$ and every positive multiple $K'$ of $K$ such that $\left[\frac{1}{K}h(n)\right] \in \mathbf W_{E',K'}$, and for every arc $U$ of length at least $\gamma$ (if $h$ is of type 1-10) or $r \in \mathbb T$ of order $k$ (if $h$ is of type 11), there exist $\epsilon'\leq\epsilon, \frac{\epsilon^*}{K'}$ and an $(E', K', \epsilon')$-arc homomorphism $\phi' < \phi$ such that:

\begin{itemize}
\item If $h$ is of one of the types 1 to 10, $\hat{\phi'}\left(\left[\frac{1}{K}h(n)\right]\right)\subseteq U$.
\item If $h$ is of type 11, $\left(\phi'\right)^0\left(\left[h(n)\right]\right)=r$.
\end{itemize}\qed
\end{teo}

We will improve Theorem \ref{arc_extension_new} so that $\epsilon'$ will no longer depend on $U$ and $\phi$, but only in the size of $U$. First, we prove a lemma.

\begin{lem}\label{lemma_finite_homomorphisms}Given $E\in [\mathfrak c]^{<\omega}$, $K$ a positive integer, $\epsilon$ a positive real, there exists a finite set $\mathcal S$ of $(E, K, \frac{\epsilon}{2})$-homomorphisms such that for every $(E, K, {\epsilon})$-homomorphism $\phi$ there exists $\psi \in \mathcal S$ such that $\psi<\phi$.
\end{lem}
\begin{proof}
Let $T\subseteq \mathbb T$ be a finite $\frac{\epsilon}{4}$-dense subset. Let $J^0$ be the set of homomorphisms from $\mathbf W_{E^0, K}$ to $\mathbb T$, which is finite, and let $J^1$ be the set of all functions from $E^1$ to $T$, which is also finite. Let $J=J^0\times J^1$. For each $j=(f_j, z_j) \in J$, let $\phi_j$ be the $(E, K, \frac{\epsilon}{2})$-arc homomorphism such that $\phi_j^0=f_j$ and $\phi_j^1(\frac{1}{K}\chi_\xi)$ be the arc centered in $z_j(\xi)$ of length $\frac{\epsilon}{2}$. Let $\mathcal S=\{\phi_j: j \in J\}$.
\end{proof}

Now we prove a new version Theorem \ref{arc_extension_new} that we need later.

\begin{teo}\label{arc_extension_revisited}
Let the following be given:
\begin{enumerate}[label=\roman*)]
    \item $E\in [\mathfrak c]^{<\omega}$,
    \item $K$ a positive integer,
    \item $\epsilon^*,\epsilon$ positive reals with $\epsilon<\frac{1}{2}$,
    \item $h$ a sequence of type 1 to 11,
    \item $\gamma>0$ if $h$ is of one of the first 10 types,
    \item if $h$ is of type 11, let $k$ be the order of $h$.
\end{enumerate}

Then there exists a cofinite set $A\subseteq \omega$ such that for every $n \in A$, for every finite $E'\supseteq E$ and positive multiple $K'$ of $K.k$ such that $\left[\frac{1}{K}h(n)\right] \in \mathbf W_{E',K'}$, there exists a positive real $\epsilon'\leq\epsilon, \frac{\epsilon^*}{K'}$ such that for every $(E, K, \epsilon)$-arc homomorphism $\phi$:

\begin{itemize}
\item if $h$ is of one of the first ten types, then for every arc $U$ of length $\geq\gamma$ there exists an $(E', K', \epsilon')$-arc homomorphism $\phi' < \phi$ such that $\hat{\phi'}\left(\left[\frac{1}{K}h(n)\right]\right)\subseteq U$.

\item if $h$ is of type 11,  then for every $r \in \mathbb T$ of order $k$ there exists an $(E', K', \epsilon')$-arc homomorphism $\phi' < \phi$ such that $(\phi')^0\left(\left[h(n)\right]\right)=r$.

\end{itemize}
\end{teo}

\begin{proof}

 We apply Lemma \ref{lemma_finite_homomorphisms} to obtain a finite set $\mathcal S$ of $(E, K, \frac{\epsilon}{2})$-homomorphisms such that for every $(E, K, {\epsilon})$ homomorphism $\phi$ there exists $\psi \in \mathcal S$ such that $\psi<\phi$.
 
 First, suppose $h$ is of one of the first 10 types.

 We obtain $A_\psi$ by applying  Theorem \ref{arc_extension_new} using $\psi$ instead of $\phi$, $\frac{\epsilon}{2}$ instead of $\epsilon$ and $\frac{\gamma}{2}$ instead of $\gamma$. Now fix $E', K'$. It follows that for each $n \in A_\psi$:

    (*) for every open arc $V$ of length $\geq \frac{\gamma}{2}$ there exists $\epsilon'_{n, V, \psi}\leq\frac{\epsilon}{2}, \frac{\epsilon^*}{K'}$ and an $(E', K', \frac{\epsilon'_{n, V, \psi}}2)$-arc homomorphism $\phi' < \psi$ such that  $\hat{\phi'}\left(\left[\frac{1}{K}h(n)\right]\right)\subseteq V$.\medskip

Let $A=\bigcap\{A_\psi: \psi \in \mathcal S\}$.

Let $T$ be a finite $\frac{\gamma}{4}$-dense subset of $\mathbb T$. Let $\mathcal V$ be the (finite) set of the arcs of length $\frac{\gamma}{2}$ centered in a point of $T$.

Fix $n\in A$. We apply (*) for each $V \in \mathcal V$ and $\psi \in \mathcal S$ to obtain $\epsilon'_{n, V, \psi}$ for which  there exists an $(E', K', \frac{\epsilon'_{n, V, \psi}}2)$-arc homomorphism $\phi_{n, V, \psi}' < \psi$ such that $\hat{\phi}_{n, V, \psi}'\left(\left[\frac{1}{K}h(n)\right]\right)\subseteq V$. 

Let $\epsilon_n'=\min\left\{\frac{\epsilon'_{n, V, \psi}}2: V \in \mathcal V, \psi \in \mathcal S\right\}$.

Now given $\phi$ and $U$ as in the statement of this theorem, there exists $\psi \in \mathcal S$ such that $\psi<\phi$, and $V \in \mathcal V$ such that $V\subseteq U$. We shrink the size of the arc homomorphism $\phi_{n, V, \psi}'$ to $\epsilon_n'$ and the proof is complete.

 Now suppose $h$ is of type 11.

 We obtain $A_\psi$ by applying  Theorem \ref{arc_extension_new} using $\psi$ instead of $\phi$ and $\frac{\epsilon}{2}$ instead of $\epsilon$. Now fix $E', K'$. It follows that for each $n \in A_\psi$:

    (*) for every $r \in \mathbb T$ of order $k$ there exists $\epsilon'_{n, r, \psi}\leq\frac{\epsilon}{2}, \frac{\epsilon^*}{K'}$ and an $(E', K', \frac{\epsilon'_{n, r, \psi}}2)$-arc homomorphism $\phi' < \psi$ such that  $(\phi')^{0}\left(\left[h(n)\right]\right)=r$.\medskip

Let $A=\bigcap\{A_\psi: \psi \in \mathcal S\}$.

Fix $n\in A$. We apply (*) for each $V \in \mathcal V$ and $\psi \in \mathcal S$ to obtain $\epsilon'_{n, r, \psi}$ for which  there exists an $(E', K', \frac{\epsilon'_{n, V, \psi}}2)$-arc homomorphism $\phi_{n, V, \psi}' < \psi$ such that  $(\phi')^{0}\left(\left[h(n)\right]\right)=r$. 

Let $\mathbb T_k$ be the set of all elements of $\mathbb T$ of order $k$, which is finite.
Let $\epsilon_n'=\min\left\{\frac{\epsilon'_{n, r, \psi}}2: r \in \mathbb T_k, \psi \in \mathcal S\right\}$.

Now given $\phi$ and $r$ as in the statement of this theorem, there exists $\psi \in \mathcal S$ such that $\psi<\phi$. We shrink the size of the arc homomorphism $\phi_{n, r, \psi}'$ to $\epsilon_n'$ and the proof is complete.
\end{proof}

We will fix  the objects that are obtained through the application of Lemma \ref{arc_extension_revisited} and give them a notation.

\begin{defin}For each $E \in [\mathfrak c]^{<\omega}$, $K$ positive integer, $\epsilon, \epsilon^*$ positive reals less than $1/2$, $h$ sequence of type 1-11 and $\gamma>0$, let $A(E, K, \epsilon, \epsilon^*, h, \gamma)$ be a fixed set $A$ obtained by applying the previous theorem ($\gamma$ does not matter if $h$ is of type 11).

Moreover, for each $n \in  A(E, K, \epsilon, \epsilon^*, h, \gamma)$, for every $E'\supseteq E$ and for every positive multiple $K'$ of $K.k$ satisfying $\left[\frac{1}{K}h(n)\right] \in \mathbf W_{E',K'}$, we let $\mathcal E(E, K, \epsilon, \epsilon^*, h, \gamma, n, E', K')$ be a fixed positive real $\epsilon'$ such that $\epsilon'\leq \epsilon, \epsilon^*/K'$ and such that for every $(E, K, \epsilon)$-homomorphism $\phi$:

\begin{itemize}
\item if $h$ is of one of the first ten types, then for every arc $U$ of length $\geq\gamma$ there exists an $(E', K', \epsilon')$-arc homomorphism $\phi' < \phi$ such that $\hat{\phi'}\left(\left[\frac{1}{K}h(n)\right]\right)\subseteq U$.
\item if $h$ is of type 11,  then for every $r \in \mathbb T$ of order $k$ there exists an $(E', K', \epsilon')$-arc homomorphism $\phi' < \phi$ such that $(\phi')^0\left(\left[h(n)\right]\right)=r$.
\end{itemize}
\end{defin}

\section{Homomorphisms on a countable subgroup} \label{section.countable.hom}

As in the previous sections, throughout this section we fix a nice subgroup $\mathbf W$ of $\mathbb W$. We will show how to construct useful homomorphisms on countable subgroups of $\mathbf W$. For the next proposition, given $a \in \mathbf Q$, we define:

$$\|a\|=\sum_{\xi \in \supp a}|a(\xi)|.$$

\begin{prop} \label{countable.hom}
Let $E$ be an infinite countable subset of $\mathfrak c$ with $\omega\subseteq E$, $e \in \mathbf{W}_E$ with $e\neq 0$, a countable $\mathcal G\subseteq \mathcal H$ (where $\mathcal H$ is defined with respect to $\mathbf{W}_E$) and $(p_g:\, g \in \mathcal G)$ a family of pairwise incomparable selective ultrafilters.

Fix a family $(c_g: g \in \mathcal G)$ of elements of $\mathbf{Q}_E$ such that $[c_g] \in \mathbf W_E$,  $c_g$ is a non torsion element if $g$ is of one of types from 1 to 10, and $[c_g]$ has the same order as $[g]$ if $g$ is of type 11.

Then there exists a homomorphism $\rho:\, \mathbf{W}_{E} \to \mathbb T$ such that:

\begin{enumerate}
\item $\rho (e) \neq 0+\mathbb Z$,
\item  $p_g$-$\lim(\rho([g(n)]):n \in \omega)=\rho([c_g])$, for each $g \in \mathcal G$, and
\item $\left(\rho\left(\frac{n!}{S}\chi_n\right): n \in \omega\right)$ converges to $0+\mathbb Z$, for every integer $S>0$. \qed
\end{enumerate} 
\end{prop}

Before we prove the proposition it is easy to produce the objects below satisfying conditions (i)-(xvi):

Enumerate $\mathcal G=\{g_m:\,  m\in \omega \}$  faithfully.

Let $E_0$ be a finite subset of $E$, $K_0$ be a positive integer, $\epsilon_0>0$ and $A_0$ be a cofinite subset of $\omega$ such that 

\begin{enumerate}[label=(\roman*)]

\item \label{5.6.i} $e, [c_{g_0}], [g_0(0)] \in \mathbf{W}_{E_{0}, K_{0}}$;

\item  \label{5.6.ii}$\epsilon_{0}\| e\| K_0<\frac{1}{2}$; 

\item $A_{0}=\omega \setminus \{0\}$;
\end{enumerate}
The three items above comprise a sufficient condition for the existence of an $(E_{0},K_{0},\epsilon_{0})$-arc homomorphism $\sigma$ such that $0\notin \hat\sigma(e)$, which will be used to prove condition $(1)$ for $\rho$.

Let $(E_t:\, t \in \omega)$ be an arbitrary sequence of finite subsets of $E$ such that 
\begin{enumerate}[label=(\roman*)]
\setcounter{enumi}{3}
\item $\bigcup_{t \in \omega}E_t=E$; 

\item  \label{5.6.viii}$E_{t+1}\supseteq E_{t}$, for every $t \in \omega$;

\item $\{[c_{g_{t}}]\}\cup\bigcup\{[g_m(m')]:m, m'\leq t\} \subset \mathbf{W} _{E_{t}}$, for every $t\in \omega$;
\end{enumerate}

Fix $(K_t:\, t \in \omega )$ an arbitrary sequence of integers so that:
\begin{enumerate}[label=(\roman*)]
\setcounter{enumi}{6}
\item $\left(K_t.\prod_{i\leq t}o(g_i)\right) \mid K_{t+1}$ and  $t! \mid K_t$, for every $t \in \omega$, where $o(g_i)$ is the order of $g_i$ if $g_i$ is of type 11, and 1 otherwise;

\item $\{[c_{g_{t}}]\}\cup\bigcup\{\left[\frac{1}{K_{t-1}}g_m(m')\right]:m, m'\leq t\} \subseteq \mathbf{W}_{E_{t},K_t}$, for every $t\in \omega \setminus \{0\}$;
\end{enumerate}

Recursively, we define sequences  $(\epsilon_t: t \in \omega \setminus \{0\})$, $(\epsilon_t^*: t \in \omega \setminus \{0\})$ of positive real numbers less than $\frac{1}{2}$ and a sequence of cofinite sets $(A_t: t \in \omega)$  satisfying:
\begin{enumerate}[label=(\roman*)]
\setcounter{enumi}{8}

\item $\epsilon_t^*= \frac{1}{2^{t+1}\max \{\|c_{g_l}\|: l\leq t+1\}}$;

\item  \label{5.6.x} $\epsilon_t\leq  \min\{ \mathcal E(E_{i}, K_{i}, \epsilon_{i}, \epsilon_{i}^*, g_j, \epsilon_{i}, n, E_t, K_t):\, j\leq i \leq t-1 \text{ and } n \in A_i \cap (t+1)\}$;

\item $A_t=\cap _{i\leq t}A(E_t, K_t, \epsilon_t, \epsilon_t^*, g_i, \epsilon_t) \setminus (t+2)$.
\end{enumerate}

Denote the ultrafilter $p_{g_k}$ as $p_k$.
The family $\{A_t:\, t \in \omega \} \subseteq p_k$, for each $k \in \omega$.
By the selectivity of $p_k$ there exists a strictly increasing $(a_{k,i})_{ i \in \omega}$ such that 
\begin{enumerate}[label=(\roman*)]
\setcounter{enumi}{11}
\item $a_{k,i} \in A_i$ and $i<a_{k,i}$, for each $i\in \omega$ and

\item $\{a_{k,i}:\, i \in \omega \} \in p_k$, for each $k\in \omega$.\\
\end{enumerate}
Applying Lemma \ref{countable.ult}, we have a family $(I_k:\, k\in \omega)$ of pairwise disjoint subsets of $\omega$ such that:
\begin{enumerate}[label=(\roman*)]
\setcounter{enumi}{13}
\item $\{a_{k,i}:\, i\in I_k\} \in p_k$ and 

\item $([i,a_{k,i}]:\, k\in \omega \text{ and } i \in I_k)$ is a family of pairwise disjoint sets. 
\end{enumerate}

Finally, we may also suppose, by removing finitely many elements of each $I_k$, that:

\begin{enumerate}[label=(\roman*)]
\setcounter{enumi}{15}
\item $I_k\cap(k+1)=\emptyset$ for every $k \in \omega$.
\end{enumerate}

\begin{proof}(of Proposition \ref{countable.hom})

Let $I=\bigcup_{k\in \omega} I_k$. Let $(i_n:\, i \in \omega)$ be an increasing enumeration of $I$ and let $k_n \in \omega$ be the unique $k$ such that $i_n \in I_k$ and set $a_n=a_{k_n,i_n}$. Notice that by (xii) and (xvi), for each $n$, $k_n\leq i_n<a_n$.

Fix $\phi'$ any $(E_{0}, K_{0}, \epsilon_{0})$-arc homomorphism such that 
$0 \notin \cl \hat \phi'(e)$. The existence of such arc homomorphism follows from conditions \ref{5.6.i} and \ref{5.6.ii}.  Since $i_0> 0$, we can fix an $(E_{i_0}, K_{i_0}, \epsilon_{i_0})$-arc homomorphism $\phi_{i_0}$ such that $\phi_{i_0}< \phi'$. From the inequality, it follows that:

\begin{enumerate}[label=(\alph*)]

\item $0 \notin \cl \hat \phi_{i_0}(e) $.
\end{enumerate}

Since $k_0 \leq i_0<a_0$ and $a_0 \in A_{i_0}$, we can find an $(E_{a_0},K_{a_0},\mathcal E(E_{i_0},K_{i_0},\epsilon_{i_0},\epsilon_{i_0}^*,g_{k_0},\epsilon_{i_0},a_0, E_{a_0},K_{a_0}))$-arc homomorphism
$\phi_*$ with $\hat \phi_* ([g_{k_0}(a_0)]) \subseteq \hat \phi_{i_0}([c_{g_{k_0}}])$ and  $\phi_*<\phi_{i_0}$.

Since $\epsilon_{a_0} \leq \mathcal E(E_{i_0},K_{i_0},\epsilon_{i_0},\epsilon_{i_0}^*,g_{k_0},\epsilon_{i_0},a_0, E_{a_0},K_{a_0})$,

there exists an $(E_{a_0}, K_{a_0}, \epsilon_{a_0})$-arc homomorphism $\phi_{a_0} < \phi_*$, thus $\phi_{a_0}< \phi_{i_0}$.

Set $B=\{i_n,a_n:\, n \in \omega\}=\{i_n:\, n \in \omega\}\dot\cup\{a_n:\, n \in \omega\}$ (notice that $i_0<a_0<i_1<a_1<\dots$).
We will show that
there exists a sequence $(\phi_{b}:\,  b\in B )$ such that

\begin{enumerate}[label=(\alph*)]
\setcounter{enumi}{1}

\item $\phi_{b}$ is an $(E_{b}, K_{b}, \epsilon_{b})$-arc homomorphism, for each $b \in B$,

\item  $ \phi_{i_{n+1}}< \phi_{a_n}< \phi_{i_n}$, for each $n \in \omega$, and 

\item   $\hat \phi_{a_n}([g_{k_n}(a_n)])\subseteq \hat \phi_{i_n}([c_{g_{k_n}}])$, for each $n\in \omega$.

\end{enumerate}

 Suppose that we have defined $\phi_b$ for $b \in \{ i_n,a_n:\, n \in N\}$
 satisfying the conditions $(b)-(d)$.
 
 It follows from $a_{N-1}<i_{N}$ that we can define a $(E_{i_{N}},K_{i_{N}},\epsilon_{i_{N}})$ arc homomorphism $\phi_{i_{N}}$ such that $\phi_{i_{N}}<\phi_{a_{N-1}}$. We can repeat the procedure that defined $\phi_{a_0}< \phi_{i_0}$ to obtain $\phi_{a_{N}}$ an $(E_{a_{N}},K_{a_{N}},\epsilon_{a_{N}})$ arc homomorphism such that $\phi_{a_{N}}<\phi_{i_N}$ and  $\hat \phi_{a_{N}}([g_{k_{N}}(a_{N})])\subseteq \hat \phi_{i_{N}}([c_{g_{k_N}}])$.
 Then the conditions $(b)-(d)$ are satisfied for $N$.\\

By Lemma \ref{homomorphismconstructor} and condition $(c)$, let $\rho$ be the homomorphism obtained from the intersection of the sequence $(\phi_{b}:\, b \in B)$. By $(b)$ and $(iv)$ and $(vii)$, it follows that the domain of $\rho$ is $\mathbf W_{E}$. Since $\omega \subseteq E$, it follows from Lemma \ref{homomorphismconstructor} that condition $(3)$ holds.

By $(a)$ and $\rho(e) \in \cl \hat \phi_{i_0}(e)$, it follows that $\rho(e)\neq 0$ and $(1)$ is satisfied. 

Now for each $k \in \omega$ and $i \in I_k$, let $n \in \omega$ be the unique natural such that 
$i=i_n$ and $k=k_n$. By $(d)$, it follows that $\rho([g_{k}(a_n)]), \rho([c_{g_k}]) \in \hat\phi_{i}([c_{g_k}])$. Since the arc $\hat\phi_{i}([c_{g_k}])$ has length at most $K_{i_n}\|c_{g_k}\|\epsilon_{i_n}\leq\|c_{g_k}\|\epsilon^*_{i_n-1}\leq 2^{-i_n}$, it follows that the distance between $\rho([g_{k}(a_n)])$ and $ \rho([c_{g_k}])$ is at most $2^{-i}$. Thus, since $(i_n:n \in \omega)$ is an increasing enumeration and recalling $a_n=a_{k_n,i_n}$, the 
sequence $(\rho([g_k(a_{k,i})]):\, i \in I_k)$ converges to $\rho([c_{g_k}])$.
By (xiv), $(2)$ is satisfied.
\end{proof}

\section{The immersion and the proof of the main theorem} \label{section.proof.main}

In this section, we consider an Abelian group $G$ such that $|G| = |G / T(G)| = \mathfrak{c}$.

\begin{defin}\label{niceimmerson}
Let  $G\subseteq \mathbb W$ be an Abelian group  such that $|G| = r(G) = \mathfrak{c}$ and for all $d, n \in \mathbb{N}$ with $d \mid n$, the group $dG[n]$ is either finite or has cardinality $\mathfrak{c}$.

 Let $D$ be the set of all integers $n > 1$ such that $G$ contains an isomorphic copy of the group $\mathbb{Z}_{n}^{(\mathfrak{c})}$.
 
Let $\mathbf W$ be a nice subgroup of $\mathbb W$.
 
 We say that  $G$ is {\em nicely immersed in} $\mathbf{W}$ if $G$ is a subgroup of $\mathbf W$ and there exist $L_1 \in [P_1]^{\mathfrak c}$ with $\omega \subseteq L_1$, a family $(L_n: n \in D)$ of pairwise disjoint elements of $[P_0]^{\mathfrak c}$ and a family $(y_\xi: \xi \in \bigcup_{n \in D} L_n)$ such that:

\begin{enumerate}[label=\alph*)]
\item $\{(0, \chi_{\xi}) \in \mathbf{W} : \xi \in L_1\} \subseteq G$,
\item $\{(y_{\xi}, 0) \in \mathbf{W} : \xi \in \cup_{n \in D} L_n\} \subseteq G$,
\item $\ordem(y_{\xi}) = n \ \forall \xi \in L_n, n \in D$ and
\item $\supp y_{\xi} \subseteq \{\xi\} \times \omega \ \forall \xi \in \cup_{n \in D} L_n$.
\end{enumerate}

In this case, we say $(G, \mathbf W, L, y)$ is a nicely immersed group. When no confusion may arise, we just say $G$ is a nicely immersed group.
\end{defin}

The following proposition appears in \cite{bellini&boero&castro&rodrigues&tomita} (Proposition 6.1) with the definition of nice immersion stated in it.

\begin{prop} \label{megazord}
Let  $G$ be an Abelian group  such that $|G| = r(G) = \mathfrak{c}$.

Then there exist $\mathbf{W}$ a nice subgroup of $\mathbb W$ and a group monomorphism $\varphi: G \to \mathbf{W}$ such that $\varphi[G]$ is nicely immersed in $\mathbf W$.
\end{prop}

This proposition allow us to consider only subgroups of $\mathbb W$ which are nicely immersed.

\begin{defin}
Given a nicely immersed group $G=(G, \mathbf W, L, y)$, a frame for a countably compact group topology for $G$ is a pair $(J, h)$, where:
\begin{itemize}
    \item $J=(J_n: n \in \{1\}\cup D\}$, with $J_n \in [L_n]^{\mathfrak{c}}$ for each $n \in \{1\}\cup D$;
    \item $h=(h_{\xi} : \xi \in J_1 \cup \bigcup_{n \in D} J_n)$ is a enumeration of $\mathcal H$ so that each element of $\mathcal H$ appears $\mathfrak c$-many times;
    \item if $\xi \in J_1$, then $h_{\xi}$ is of type $i \in \{1, \ldots, 10\}$;

     \item if there exists $k \in D$ such that $\xi \in J_k$, then $h_{\xi}$ is of type 11 and $\ordem([h_{\xi}(n)]) = k$, for every $n \in \omega$;

     \item $\bigcup_{n \in \omega} \supp h_{\xi}(n) \subset (\xi \times \omega) \cup \xi$, for every $\xi \in J_1 \cup \bigcup_{n \in D} J_n$ and
       
       \item $L_1 \setminus J_1$ has cardinality $\mathfrak c$.
\end{itemize}
\end{defin}

It is easy to verify that given a nicely immersed group $G$, there exist frames for a countably compact group topology for it.

\begin{defin} Given a frame $(J, h)$ for a nicely immersed countably compact group $G=(G, \mathbf W, L, y)$, we say that a group monomorphism $\Phi:\, \mathbf W \rightarrow \mathbb T^\mathfrak c$ and a family of free ultrafilters $(p_\xi:\, \xi \in J_1\cup \bigcup_{n\in D}J_n)$
fulfill the frame if:

\begin{enumerate}
\item if $\xi \in J_1$, then the $p_\xi$-limit of sequence $(\Phi([h_{\xi}(n)]) : n \in \omega)$ is $\Phi(0, \chi_{\xi})$;

  \item if $\xi \in \bigcup_{n \in D} J_n$, then the $p_\xi$-limit of the sequence $(\Phi([h_{\xi}(n)]) : n \in \omega)$ is $\Phi(y_{\xi}, 0)$;

  \item for each $S$ positive integer, the sequence $\left( \Phi \left( 0, \frac{1}{S}   n!   \chi_{n} \right) : n \in \omega \right)$ converges to $0 + \mathbb{Z}$.
\end{enumerate}
\end{defin}

\begin{teo} \label{teo.ful.conv.seq}
Let $G$ be a nicely immersed group and $(J, h)$ be a frame for it. If $\Phi$ and $(p_\xi:\, \xi \in J_1\cup \bigcup_{n\in D}J_n)$  fulfill the frame, then $\Phi[G]$ generates a countably compact group topology on $G$ with non-trivial convergent sequences, so $G$ admits such a topology.
\end{teo}

\begin{proof}
Let $G$ and $(J, h)$ be given as in the statement of the proposition. Let $g: \omega \to \Phi[G]$. We have that $\Phi^{-1}\circ g:\omega \to G$, so take any $\tilde g:\omega\to\mathbf Q$ such that $[\tilde g]=\Phi^{-1}\circ g.$ It follows from Proposition \ref{prop_subseq} that there exist $h:\omega\rightarrow \mathbf Q$ such that $h\in \mathcal{H}$ or $[h]$ is constant and in $G$, $c\in \mathbf Q$ with $[c] \in G$, $F \in [\omega]^{< \omega}$, $p_i, q_i \in \mathbb{Z}$ with $q_i \neq 0$ for every $i \in F$, $(j_{i}:i\in F)$ increasing enumerations of subsets of $\omega$  and $j: \omega \to \omega$ strictly increasing such that

\[\tilde g \circ j = h + c + \sum_{i \in F} \frac{p_i}{q_i}   f\circ j_{i}\]

with $q_i\leq j_i(n)$ for each $n \in \omega$ and $i \in F$, where $f : \omega \to \mathbf Q$ is given by $f(n) = (0, n!   \chi_n)$ for every $n \in \omega$.

In the case where $[h]$ is constant, say constantly $v\in\mathbf W$, we have that $g\circ j=\Phi\circ[\tilde g]\circ j=\Phi\circ[\tilde g\circ j]=\Phi(v)+\Phi([c])+\sum_{i\in F}\Phi\circ([\frac{p_i}{q_i}f\circ j_i])$ converges to $\Phi(v)+\Phi([c])$.

Now, in the case $h\in\mathcal H$:

Since it is fulfillment, 
  if $\xi \in J_1$, then \[\Phi(0, \chi_{\xi}) = p_{\xi}-\lim (\Phi([h_{\xi}(n)]) : n \in \omega)\] for every $\alpha < \mathfrak{c}$ and if $\xi \in \bigcup_{n \in D} J_n$, then \[\Phi(y_{\xi}, 0) = p_{\xi}-\lim (\Phi([h_{\xi}(n)]) : n \in \omega)\] for every $\alpha < \mathfrak{c}$. 

Since the sequence $\left( \Phi \left( 0, \frac{1}{q_i}   n!   \chi_{n} \right) : n \in \omega \right)$ converges to $0$ for each $i \in F$, it follows from Proposition \ref{prop_p-limit} that if $\xi \in J_1$, then
\[\Phi((0, \chi_{\xi}) + [c]) = p_{\xi}-\lim \left( \Phi \left([h_{\xi}(n)] + [c] + \sum_{i \in F} \left[ \frac{p_i}{q_i}   f\circ j_i(n)\right] \right) : n \in \omega \right)\]

and if $\xi \in \bigcup_{n \in D} J_n$, then
\[\Phi((y_{\xi}, 0) + [c]) = p_{\xi}-\lim \left( \Phi \left( [h_{\xi}(n)] + [c] + \sum_{i \in F} \left[ \frac{p_i}{q_i}   f\circ j_i(n))\right] \right) : n \in \omega \right).\]

Since $\Phi\circ \left( [h_{\xi}] + [c] + \sum_{i \in F} \left[ \frac{p_i}{q_i}   f\circ j_i)\right] \right)=\Phi\circ[\tilde g\circ j]=\Phi\circ[\tilde g]\circ j=g\circ j$, we may conclude that $\Phi((0, \chi_{\xi}) + [c])$ is an accumulation point of $g$ if $\xi \in J_1$ and $\Phi((y_{\xi}, 0) + [c])$) is an accumulation point of $g$ if $\xi \in \bigcup_{n \in D} J_n$.
\end{proof}

The following result will be used later to provide countably compact group topologies without non-trivial convergent sequences using frames. 

\begin{cor} \label{cor.extracting.ccgwntcs} Let $G$ be a nicely immersed group and $(J, h)$ be a frame for it.  \label{cor.ful.non.conv.seq} If $\Phi$ and $(p_\xi:\, \xi \in J_1\cup \bigcup_{n\in D}J_n)$ fulfill the frame, then $\Phi$ generates a countably compact group topology on $G\cap \mathbf {W}_{\mathfrak c \setminus \omega}$ without non-trivial convergent sequences.
\end{cor}

\begin{proof} 
We will denote $G_1=G\cap \mathbf {W}_{\mathfrak c \setminus \omega}$. Notice that $G_1$ has size $\mathfrak c$.
Let $g: \omega \to \Phi[G_1]$. 
Since $\Phi$ is a monomorphism and thus one-to-one, we have that
$\Phi^{-1}\circ g:\omega \to G_1$ is a sequence 
in $G$.

So take some $\tilde g:\omega\to\mathbf Q_{\mathfrak c \setminus \omega}$ such that 
$[\tilde g]=\Phi^{-1}\circ g.$ Following the proof of Proposition \ref{prop_subseq}, which is found in \cite{bellini&boero&castro&rodrigues&tomita}, it is clear that as
the support of $\Phi^{-1}\circ g$ does not intersect $\omega$ and the manipulations to get the types do 
not add new support, that there exist 
$h:\omega\rightarrow \mathbf Q_{\mathfrak c \setminus \omega}$ such that 
$h\in \mathcal{H}$ or $[h]$ is constant and in $G_1$, $c\in \mathbf Q$ with $[c] \in G_1$ and $j:\omega\rightarrow \omega$ strictly increasing such that

\[\tilde g \circ j = h + c \]

as in this case $F$ that appears in the general case is empty.

In the case where $[h]$ is constant, say constantly $v\in\mathbf W_{\mathfrak c \setminus \omega}$, we have that $g\circ j=\Phi\circ[\tilde g]\circ j=\Phi\circ[\tilde g\circ j]=\Phi(v)+\Phi([c])$ converges to $\Phi(v)+\Phi([c])$.

Now, in the case $h\in\mathcal H$:

For each $\xi \in J_1$, then \[\Phi(0, \chi_{\xi}) = p_{\xi}-\lim (\Phi([h_{\xi}(n)]) : n \in \omega)\] for every $\alpha < \mathfrak{c}$ and if $\xi \in \bigcup_{n \in D} J_n$, then \[\Phi(y_{\xi}, 0) = p_{\xi}-\lim (\Phi([h_{\xi}(n)]) : n \in \omega)\] for every $\alpha < \mathfrak{c}$. 

It follows from Proposition \ref{prop_p-limit} that if $\xi \in J_1$, then
\[\Phi((0, \chi_{\xi}) + [c]) = p_{\xi}-\lim \left( \Phi \left([h_{\xi}(n)] + [c] \right) : n \in \omega \right)\]

and if $\xi \in \bigcup_{n \in D} J_n$, then
\[\Phi((y_{\xi}, 0) + [c]) = p_{\xi}-\lim \left( \Phi \left( [h_{\xi}(n)] + [c]  \right) : n \in \omega \right).\]

Now, fix $\xi \in J_1 \cup \bigcup_{n \in D} J_n$ such that $h=h_\xi$.

Then $\Phi\circ \left( [h_{\xi}] + [c]  \right)=\Phi\circ[\tilde g\circ j]=\Phi\circ[\tilde g]\circ j=g\circ j$ implies that $\Phi((0, \chi_{\xi}) + [c])$ is an accumulation point of $g$ if $\xi \in J_1$ and $\Phi((y_{\xi}, 0) + [c])$) is an accumulation point of $g$ if $\xi \in \bigcup_{n \in D} J_n$. 

Hence $G_1$ is countably compact.

Since $h$ appears more than once in the enumeration, it follows that every 1-1 sequence $g$ is not a convergent sequence. Therefore $G_1$ does not have non-trivial convergent sequences.
\end{proof}

\begin{defin}Given a nicely immersed group $G=(G, \mathbf W, L, y)$, $w \in \mathbf W$, a frame for a countably compact group topology $(J, h)$, we let $E(w, G, J, h)$ be a fixed countable set $E\subseteq \mathfrak c$ for which $\omega\subset E$, $w\in\mathbf{ W}_E$, $[h_\xi (n)] \in \mathbf{ W}_E$ for each $n\in \omega$ and  each $\xi \in E \cap ( J_1\cup \bigcup_{k \in D} J_k)$. Moreover, we require that $E\subseteq \omega \cup (\sup(\supp w)+1)$.
\end{defin}

It is easy to show that $E(w, G, J, h)$ is well defined, that is, such an set $E$ exists. Moreover notice that $E(w, G, J, h)$ should really be $E(w, (G, \mathbf W, L, y), J, h)$, but we will use this shorthand whenever it does not generate confusion.

\begin{prop}\label{prop_hom} Let $G=(G, \mathbf W, L, y)$ be a nicely immersed group, $(J, h)$ a frame, $w \in \mathbf W$ be non-zero and $E=E(w, G, J, h)$ and $(a_\mu:\, \mu \in R_1)$ be arbitrary
elements of $\mathbb T$ where $R_1$ is a subset of $L_1 \setminus J_1$ that does not intersect $E$. 

Assume the existence of $\mathfrak c$ incomparable selective ultrafilters and enumerate them faithfully as $(p_\xi:\, \xi \in J_1 \cup \bigcup_{n \in D} J_n)$.

There exists a group homomorphism $\rho : \mathbf W \to \mathbb{T}$ satisfying the following conditions:

\begin{enumerate}[label=(\roman*)]

  \item $\rho(w) \neq 0 + \mathbb{Z}$;

  \item if $\xi \in J_1$, then the $p_\xi$-limit of sequence $(\rho([h_{\xi}(n)]) : n \in \omega)$ is $\rho(0, \chi_{\xi})$;

  \item if $\xi \in \bigcup_{n \in D} J_n$, then the $p_\xi$-limit of the sequence $(\rho([h_{\xi}(n)]) : n \in \omega)$ is $\rho(y_{\xi}, 0)$;

  \item for each $S$ positive integer, the sequence $\left( \rho \left( 0, \frac{1}{S}   n!   \chi_{n} \right) : n \in \omega \right)$ converges to $0 + \mathbb{Z}$.

\item if $\xi \in R_1$ then $\rho(0, \chi_\xi)=a_\xi$;
\end{enumerate}
\end{prop}
\begin{proof}

Let $\mathcal G=\{h_\xi:\, \xi \in E \cap (J_1 \cup \bigcup_{k\in D} J_k)\}$ and $c_{h_\xi}\in\mathbf Q$ such that $[c_{h_\xi}] = (y_\xi,0)$ if $\xi \in E \cap \bigcup_{k \in D} J_k$ and $[c_{h_\xi}]=(0, \chi_\xi)$ if $\xi \in E \cap J_1$. Apply Proposition \ref{countable.hom} to obtain a homomorphism $\psi: \mathbf{ W}_E \longrightarrow \mathbb T$. Then conditions (i)-(iv) are satisfied for every $\xi \in E$. In particular, any homomorphism that extends $\psi$ will satisfy conditions (i) and (iv). 

Let $(\xi_\gamma:\gamma<\mathfrak c)$ be an increasing enumeration of $\mathfrak c\backslash E$ and let $E_\gamma=E\cup\{\xi_\alpha:\alpha<\gamma\}$ for each $\gamma<\mathfrak c$. We will define recursively an increasing sequence of homomorphisms $\psi_\gamma: \mathbf{ W}_{E_\gamma}\to\mathbb T$ for each $\gamma<\mathfrak c$ satisfying (ii) and (iii) in their domains.

So we let $\psi_0=\psi$ and $\psi_\gamma=\bigcup_{\alpha<\gamma}\psi_\alpha$ if $\gamma$ is a limit ordinal. Suppose now $\gamma=\alpha+1$. If $\xi_{\alpha}\notin R_1 \cup J_1 \cup \bigcup_{k\in D} J_k$, we extend $\psi_\alpha$ to $\psi_{\alpha+1}$ using the divisibility of $\mathbb T$.  If $\xi_\alpha \in R_1$ we define $\psi_{\alpha+1}(0,\chi_{\xi_\alpha})=a_{\xi_\alpha}$ by the divisibility of $\mathbb T$ and the infinite order of $\chi_{\xi_\alpha}$. If $\xi_\alpha\in J_1 \cup \bigcup_{k\in D} J_k$, since $[h_{\xi_\alpha}(n)] \in\mathbf {W}_{E_{\alpha}}$ for each $n\in\omega$, the $p_{\xi_\alpha}$-limit of  $(\psi_{\alpha}([h_{\xi_\alpha}(n)]) : n \in \omega)$ is some $x\in\mathbb T$. Notice that if $\xi_\alpha\in J_k$ for some $k\in D$, then $o(x)$ divides $k$. Thus using the divisibility we define $\psi_{\alpha+1}$ on $\mathbf {W}_{E_{\alpha+1}}$ satisfying: $\psi_{\alpha+1}(0,\chi_{\xi_\alpha})=x$ if $\xi_\alpha\in J_1$, and $\psi_{\alpha+1}(y_{\xi_\alpha},0)=x$ if $\xi_\alpha\in \bigcup_{k\in D} J_k$.

Define $\rho=\bigcup_{\gamma<\mathfrak c}\psi_\gamma$ and we are done.
\end{proof}

Finally, we prove the first variation of the main result in this paper. Recall that $\mathcal U$ has been defined as the basis of $\mathbb T^c$ consisting of products of nonempty arcs. If $U \in \mathcal U$, $\supp U=\{\alpha<\mathfrak c: U_\alpha\neq \mathbb T\}$, where $U=\prod_{\alpha<\mathfrak c} U_\alpha$.

\begin{teo} \label{ccgwithcs}Given a frame $(J, h)$ for a nicely immersed group $G$, let $\mathcal P=(p_\xi:\, \xi \in J_1 \cup \bigcup _{n\in D}J_n)$  be a faithful enumeration of pairwise incomparable selective ultrafilters. Then there is $\Phi$  and $(R_U:\, U \in \mathcal U)$ a family of pairwise disjoint subsets of $L_1 \setminus J_1 $ of cardinality $\mathfrak c$ such that $\Phi$ and $\mathcal P$ fulfill the frame and $\Phi(0,\chi_\xi) \in U$ for each $U\in \mathcal U$ and
$\xi \in R_U$.
\end{teo}

\begin{proof} Enumerate $\mathbf W \setminus \{0\}$ as $\{w_\alpha:\, \alpha <\mathfrak c\}$ so that $\max \supp w_\alpha < \max\{ \omega, \alpha \}$.
Then $E_\alpha=E(w_\alpha, G, J, h) \subset \max \{\omega , \alpha\}$.
Set $\{R_U:\, U \in \mathcal U\}$ a family of disjoint subsets $L_1 \setminus (J_1\cup \omega)$ of cardinality $\mathfrak c$ such that $\max \supp U <\xi$, for each $\xi \in R_U$.

For each $U \in \mathcal U$ and each $\xi \in R_U$, fix $z_\xi \in \mathbb T^{\supp U}$ such
that $z_\xi (\mu)\in U_\mu$ for each $\mu \in \supp U$, where $U_\mu$ is the $\mu$-th coordinate of $U$.

Now we aim to apply Proposition \ref{prop_hom} for each $\alpha$ as follows:

For each $\alpha$, let $R_{1,\alpha}=\bigcup \{ R_U:\, \alpha \in \supp U\}$ and set $a_\xi=z_\xi(\alpha)$ for each $\xi \in R_{1,\alpha}$.
Note that if $\xi \in R_U$ and $\alpha \in \supp U$ then $\xi >\alpha$ and therefore 
$R_{1,\alpha}\cap E_\alpha \subseteq R_{1,\alpha} \cap \max\{\omega,\alpha\} =\emptyset$. 

Thus, we can apply Proposition \ref{prop_hom} to obtain a homomorphism $\rho_\alpha$ that satisfies the $i)$-$iv)$
for every $\alpha < \mathfrak c$ and $v)$ for $R_{1,\alpha}$.

\textbf{Claim:} for each $U\in \mathcal U$, $\xi \in R_U$, and $\alpha \in \supp U$, we have
$\rho_\alpha(0,\chi_\xi) \in U_\alpha$.

Indeed, given $U$, $\xi$ and $\alpha$, we have
$\xi \in R_{1,\alpha}$ and $\rho_\alpha(0,\chi_\xi)=z_\xi(\alpha)\in U_\alpha$.

It follows from Proposition \ref{prop_hom} (i) and that every non-zero element of $\mathbf W$ is 
enumerated by some $w_\alpha$ that \[\begin{array}{cccc}
                                                        \Phi : & \mathbf W & \to & \mathbb{T}^{\mathfrak{c}} \\
                                                               & w & \mapsto & \Phi(w)
                                                      \end{array}
\] given by $\Phi(w)(\alpha) = \rho_{\alpha}(w)$ for every $\alpha < \mathfrak{c}$ is a group monomorphism. Thus, $\Phi[G]$ is isomorphic to $G$ and since $\mathbb{T}^{\mathfrak{c}}$ is a Hausdorff topological group, the subspace topology induced by $\mathbb{T}^{\mathfrak{c}}$ turns $\Phi[G]$ into a Hausdorff topological group and, by the claim, $\Phi(0,\chi_\xi) \in U$ for each $U\in \mathcal U$ and $\xi \in R_U$.

It follows from Proposition \ref{prop_hom} (iv) that if $S$ is a positive integer, then the sequence $\left( \Phi \left( 0, \frac{1}{S}   n!   \chi_{n} \right) : n \in \omega \right)$ converges to $0$.

 According to Proposition \ref{prop_hom} (ii) and (iii), 
  if $\xi \in J_1$, then \[\rho_{\alpha}(0, \chi_{\xi}) = p_{\xi}-\lim (\rho_{\alpha}([h_{\xi}(n)]) : n \in \omega)\] for every $\alpha < \mathfrak{c}$ and if $\xi \in \bigcup_{n \in D} J_n$, then \[\rho_{\alpha}(y_{\xi}, 0) = p_{\xi}-\lim (\rho_{\alpha}([h_{\xi}(n)]) : n \in \omega)\] for every $\alpha < \mathfrak{c}$. It follows from Proposition \ref{prop_p-limit_product} that $\xi \in J_1$, then \[\Phi(0, \chi_{\xi}) = p_{\xi}-\lim (\Phi([h_{\xi}(n)]) : n \in \omega)\] and if $\xi \in \bigcup_{n \in D} J_n$, then \[\Phi(y_{\xi}, 0) = p_{\xi}-\lim (\Phi([h_{\xi}(n)]) : n \in \omega).\]

Therefore, $\Phi$ and $\mathcal P$ is a fulfillment of the frame.
\end{proof}

\begin{lem}\label{RemoveOmega}Let $(G, \mathbf W, L, y)$ be a nicely immersed group in $\mathbb W$. Then there exists a nicely immersed group $(G^*, \mathbf W^*, L^*, y^*)$ such that $G$ is isomorphic to $G^*\cap \mathbf W^*_{\mathfrak c\setminus \omega}$.
\end{lem}

\begin{proof}
Let us assume $(G, \mathbf W, L, y)$ is a nicely immersed subgroup of $\mathbb W$.

Let $I:\mathfrak c\rightarrow (\mathfrak c\setminus \omega)$ be a function satisfying:

\begin{enumerate}
\item $I$ is injective and onto $(\mathfrak c\setminus \omega)$, and
\item $I[P_0]=P_0$, $I[P_1]=P_1\setminus \omega$.
\end{enumerate}
It is easy to see that such an $I$ exists.

Let $u:\mathbb W\rightarrow \mathbb W$ be the $\mathbb Q$-linear mapping defined by $u(\chi_\xi)=\chi_{I(\xi)}$, if $\xi \in P_1$, and $u(\chi_{(\xi, n)})=\chi_{(I(\xi), n)}$, if $\xi \in P_0$, $n \in \omega$.

Let $G'=u[G]$, $\mathbf W'=u[\mathbf W]$.

Let $G^*=G'\oplus \mathbf W_{\omega}$, $\mathbf W^*=\mathbf W'\oplus \mathbf W_{\omega}$, $L_1^*=I[L_1]\cup \omega$, $L_{n}^*=I[{L_n}]$ for $n \in D$ and $y_{I(\xi)}^*=u(y_\xi)$ for $\xi \in \bigcup_{n \in D}L_n$. Then $(G^*, \mathbf W^*, L^*, y^*)$ is a nicely immersed group, and $G^*\cap \mathbf W^*_{\mathfrak c\setminus \omega}=G'\approx G$.

\end{proof}

\begin{cor} \label{equiv} Assume the existence of $\mathfrak c$ selective ultrafilters.
Let $G$ be a non-torsion Abelian group of size continuum. Then the following are equivalent:

\begin{enumerate}

\item the free rank of $G$ is equal to $\mathfrak{c}$ and, for all $d, n \in \mathbb{N}$ with $d \mid n$, the group $dG[n]$ is either finite or has cardinality $\mathfrak{c}$;

    \item $G$ admits a countably compact Hausdorff group topology; 
    \item $G$ admits a countably compact Hausdorff group topology with non-trivial convergent sequences; and
    
    \item $G$ admits a countably compact group topology without non-trivial convergent sequences.
\end{enumerate}\qed
\end{cor}

\begin{proof} It was proved in \cite{dikranjan&tkachenko} that $(2)$
implies $(1)$ in $ZFC$. That $(3)$ implies $(2)$ is trivial. That $(1)$ implies $(3)$ is Theorem \ref{ccgwithcs}.

That condition $(4)$ implies condition $(2)$, is obvious.

 $(1)$ implies $(4)$ follows from  Corollary \ref{cor.extracting.ccgwntcs} and Lemma \ref{RemoveOmega}.

\end{proof}

\section{More examples} \label{moreexamples}

\subsection{ Algebraic countably compact groups}

The first examples of countably compact groups of cardinality greater than $\mathfrak c$ used forcing  \cite{koszmider&tomita&watson}, which later was adapted to obtain a countably compact group whose cardinality has countable cofinality \cite{tomita4}, \cite{tomita5}, \cite{castro-pereira&tomita}. Dikranjan and Shakhmatov \cite{dikranjan&shakhmatov2} obtained a consistent classification of all Abelian groups of cardinality at most $2^{\mathfrak c}$ that admit a countably compact group topology. The forcing examples obtained so far satisfy $CH$ and only work for groups of cardinalilty at most $2^\mathfrak c$. 

The use of selective ultrafilters for the construction of countably compact groups first appeared in \cite{tomita&watson} and they were afterwards used to construct other examples in \cite{garcia-ferreira&tomita&watson}, \cite{castro-pereira&tomita2010}, \cite{madariaga-garcia&tomita}, \cite{tomita2015}, \cite{tomita3}, \cite{boero&garcia-ferreira&tomita}.

Recall one of the equivalent definitions of an algebraically countably compact group \cite{fuchs1970infinite}.

\begin{defin} An Abelian group $G$ is \textit{algebraically countably compact} if there exists an Abelian group $H$ such that the  group $G\oplus H$ admits a countably compact group topology.
\end{defin}

As a consequence of their classification, Dikranjan and Tkachenko showed that (under Martin's Axiom) every Abelian group of cardinality $\mathfrak c$ is algebraically countably compact. 

 From the classification in \cite{dikranjan&shakhmatov2} it also follows that every Abelian group of cardinality at most $2^{\mathfrak c}$ is algebraically countably compact. However, this forcing model is under $CH$.

The existence of $2^\mathfrak c$ selective ultrafilters follows from $CH$ or Martin's Axiom.
As a corollary to our construction, we will  show that assuming the existence of $2^\mathfrak c$ incomparable selective ultrafilters, every Abelian group $G$
of cardinality at most $2^\mathfrak c$ is algebraically countably compact. For this, it suffices to construct the following example.

\begin{ex} \label{algebraically.cc} Assume the existence of $2^{\mathfrak c}$ selective ultrafilters and let $G$ be an arbitrary Abelian group of cardinality at most $2^{\mathfrak c}$. 
Then $G \oplus (\mathbb Q / \mathbb Z)^{(2^\mathfrak c)}\oplus \mathbb Q^{(2^\mathfrak c)}$ admits a countably compact group topology.
\end{ex}

\begin{proof} Let $\mathcal P$ be a family of incomparable selective ultrafilters of cardinality $2^\mathfrak c$ and let $I_0$, $I_1$, $I_2$ and $I_3$ be a partition of $2^\mathfrak c$ into sets of cardinality $2^\mathfrak c$ with $\omega \subseteq I_3$. Define $P_0=I_0 \cup I_2$ and $P_1=I_1 \cup I_3$.
 Set  ${\mathbf W}=(\mathbb Q/\mathbb Z)^{(P_0)}\oplus \mathbb Q^{(P_1)}$.

 Then $G$ can be embedded in $(\mathbb Q/\mathbb Z)^{(I_0)}\oplus \mathbb Q^{(I_1)}$, so wlog, $G$ is a subgroup of the latter.
Let $G_*=G \oplus (\mathbb Q/\mathbb Z)^{(I_2)}\oplus \mathbb Q^{(I_3)}$, which is a subgroup of $\mathbf W$.

 For $G_*$,  its correspondent $D$ is $ \omega\setminus 2$ (see definition \ref{niceimmerson}). Fix  $L_1 \in [I_3]^{2^\mathfrak c}$ with $\omega \subseteq L_1$ and $(L_n:\, n \geq 2 )$ a pairwise disjoint family of subsets of $I_2$ of cardinality $2^\mathfrak c$.
 Then there is trivially a family $(y_\xi: \xi \in \bigcup_{n \in D} L_n)$ such that:

\begin{enumerate}[label=\alph*)]
\item $\{(0, \chi_{\xi}) \in \mathbf{ W} : \xi \in L_1\} \subset G_*$,
\item $\{(y_{\xi}, 0) \in \mathbf{W} : \xi \in \cup_{n \in D} L_n\} \subset G_*$,
\item $\ordem(y_{\xi}) = n \ \forall \xi \in L_n, n \in D$ and
\item $\supp y_{\xi} \subset \{\xi\} \ \forall \xi \in \cup_{n \in D} L_n$.\\
\end{enumerate}

Using this immersion, we can now follow the proof for the construction of the groups of cardinality $\mathfrak c$ to produce a countably compact group topology for this $G_*$.
\end{proof}

\subsection{Countably compact groups without non-trivial convergent sequences whose weight has countable cofinality.}

\begin{lem} \label{l.largeweight} Let $G=(G,\mathbf W,L,w)$ be a nicely immersed group, $(J,h)$ be a frame for a countably compact group topology with $\omega\subseteq J_1$ and $(p_\xi:\, \xi \in \bigcup \{ J_n:\, n \in \{1\} \cup D\})$ and $\Phi$ be a fulfillment. Suppose that there exists $(R_U:\, U \in \mathcal U)$ a family of pairwise disjoint non-empty subsets $L_1 \setminus J_1$ of cardinality $\mathfrak c$ such that $\Phi(\chi_\mu) \in U$ for each $U\in \mathcal U$ and $\mu \in R_U$. Then, for every cardinal $\kappa \in [\mathfrak c, 2^\mathfrak c]$ there exists a countably compact group topology with non-trivial convergent sequences on $G$ whose weight is $\kappa$ and  a countably compact group topology without non-trivial convergent sequences on $G$ whose weight is $\kappa$.
\end{lem}

\begin{proof} 

Given $\kappa$, we can fix  $\{a_\xi:\, \xi \in L_1 \setminus J_1\}$  such that $\{a_\xi:\, \xi \in R_U \}$ is a dense subset of $\mathbb T^{\kappa}$ for each $U \in \mathcal U$. We can define a homomorphism $\Theta:\mathbf W\, \to \mathbb T^\kappa$ such that 

\begin{enumerate}[label=(\roman*)]

\item  $\Theta (0, \chi_\xi)=a_\xi$ for each $\xi \in  L_1 \setminus J_1 $;

  \item if $\xi \in J_1$, then the $p_\xi$-limit of sequence $(\Theta([h_{\xi}(n)]) : n \in \omega)$ is $\Theta(0, \chi_{\xi})$;

  \item if $\xi \in \bigcup_{n \in D} J_n$, then the $p_\xi$-limit of the sequence $(\Theta([h_{\xi}(n)]) : n \in \omega)$ is $\Theta(y_{\xi}, 0)$;

  \item for each $S$ positive integer, the sequence $\left( \Theta \left( 0, \frac{1}{S}   n!   \chi_{n} \right) : n \in \omega \right)$ is constantly $0 \in \mathbb T^\kappa$.
  \end{enumerate}
  
  Consider the topology generated by the homomorphisms $\Phi$ and $\Theta$. Consider $\Delta:\mathbf W\to \mathbb T^{\mathfrak c}\times \mathbb T^{\kappa}$ be given by $\Delta=(\Phi, \Theta)$. Then $\left(p_\xi:\, \xi \in \bigcup \{ J_n:\, n \in \{1\} \cup D\}\right)$ and $\Delta$ also fulfills the frame. By Theorem \ref{teo.ful.conv.seq}, it follows that $G$ can be equipped with a countably compact group topology with a non-trivial convergent sequences.
  
  Note that the set $\{ \Delta(\chi_\xi):\, \xi \in L_1 \setminus J_1\}$ is dense in $\mathbb {T}^{\mathfrak c} \times \mathbb{T}^{\kappa}$. Thus, the topology on $G$ also has weight $\kappa$.
  
  Now repeating the argument above in $G\oplus \mathbf{W}_\omega$, using Lemma \ref{RemoveOmega}, Corollary \ref{cor.extracting.ccgwntcs} and the argument in Corollary \ref{equiv}, we conclude that $G$ can also be equipped with a countably compact group topology without non-trivial convergent sequences that can be densely embedded  in $\mathbb {T}^{\mathfrak c} \times \mathbb{T}^{\kappa}$. Thus, the topology on $G$ also has weight $\kappa$. 
\end{proof}

Combining Lemma \ref{l.largeweight} and Theorem \ref{ccgwithcs}, we get the following result:

\begin{ex} \label{largeweight} Assume the existence of $\mathfrak c$ selective ultrafilters. Let $G$ be a group of cardinality $\mathfrak c$ that admits a countably compact group topology. Then for every cardinal $\kappa \in [\mathfrak c, 2^\mathfrak c]$ there exist countably compact group topologies with and without non-trivial convergent sequences on $G$ whose weight is $\kappa$.
\end{ex}

\begin{ex} It is consistent that every Abelian group of cardinality $\mathfrak c$ that admits a countably compact group topology admits a countably compact group topology without non-trivial convergent sequences whose weight has countable cofinality.
\end{ex}
\begin{proof}
Take a model in which there exist $\mathfrak c$ selective ultrafilters and there are cardinals of countable cofinality between $\mathfrak c$ and $2^\mathfrak c$. Apply Example \ref{largeweight} and we are done.
\end{proof}

\subsection{Non-homeomorphic countably compact group topologies on a group of cardinality $\mathfrak c$}

Tomita \cite{tomita6} showed the existence of $2^\mathfrak c$ group topologies on the free Abelian group of cardinality $\mathfrak c$ from $2^\mathfrak c$ selective ultrafilters. We will now show that the same happens to all the Abelian groups of cardinality $\mathfrak c$ that admit a countably compact group topology.

First we recall the definition in \cite{tomita6}:

\begin{defin}
If $X$ is a topological space and if $x \in X$ is an accumulation point of $s:\omega\rightarrow X$, we define $\mathcal{F}(X, s, x) $ as the set \[ \{A \subset \omega : \exists U \ \hbox{an open neighborhood of $x$ s. t.} \ \{n \in \omega : s(n) \in U\} \subset A\}.\]

We denote by $\mathcal{F}(X)$ the family consisting of all filters $\mathcal{F}(X, s, x)$, where $X$ is a fixed topological space and $x \in X$ is an accumulation point of some $s:\omega\rightarrow X$.
\end{defin}

It is straightforward to verify that $\mathcal{F}(X, s, x)$ is a filter over $\omega$.  Observe that if $|X| = \mathfrak{c}$, then $|\mathcal{F}(X)| \leq \mathfrak{c}$.

The proof of the following lemma is straightforward.

\begin{lem}\label{cap6_lem_espacos_homeomorfos_tem_mesmo_filtro}
Let $X$ and $Y$ be topological spaces and let $h: X \to Y$ be a homeomorphism. If $x \in X$ is an accumulation point of $s:\omega\rightarrow X$, then $h(x) \in Y$ is an accumulation point of $h\circ s:\omega\rightarrow Y$ and $\mathcal{F}(X, s, x) = \mathcal{F}(Y, h\circ s, h(x))$. \qed
\end{lem}

\begin{cor}\label{cor_espacos_homeomorfos_tem_mesmo_filtro}
If $X$ and $Y$ are homeomorphic, then $\mathcal{F}(X) = \mathcal{F}(Y)$. \qed
\end{cor}

\begin{lem} \label{filter} Given a frame $(J, h)$ for a nicely immersed group $G$ a $\mathcal P=(p_\xi:\, \xi \in J_1 \cup \bigcup _{n\in D}J_n)$ and $\Phi$ a fulfillment of the frame and $( R_U:\, U \in \mathcal U)$ a family of pairwise disjoint subsets of $L_1 \setminus J_1 $ of cardinality $\mathfrak c$ such that  $\Phi(\chi_\xi) \in U$ for each $U\in \mathcal U$ and $\xi \in R_U$.

Let $A=(a_n:\, n \in \omega)$ be an injective sequence in $ L_1\setminus (J_1\cup \omega)$. Let $\beta \in J_1$ be such that $h_\beta$ is the sequence given by $h_\beta (n)=\chi_{a_n} $ for each $n\in \omega$ (notice that $h_\beta$ is of type 1).

Then there exist a family $(S_W:\, W \in \mathcal U)$ of pairwise disjoint subsets of $L_1 \setminus (J_1 \cup \beta)$ of cardinality $\mathfrak c$ and $\eta: \mathbf W \rightarrow \mathbb T ^\mathfrak c$  such that $\mathcal P$ and $\eta $ are a fulfillment, $\eta (\chi_\mu)\in W$ for each $W\in \mathcal U$ and $\mu \in S_W$, $p_\beta \in \mathcal F (\langle G, \tau_\eta\rangle)$
and $p_\beta \in \mathcal F (\langle G\cap \mathbb W_{\mathfrak c\setminus \omega}, \tau_\eta|_{G\cap\mathbb W_{\mathfrak c\setminus \omega}}\rangle)$.
\end{lem}

\begin{proof}  

For each $ V \in \mathcal U$, let $z_{V} \in \mathbb T^{\supp V}$ be
such that $z_{V}(\xi)\in V_\xi$, where $V_\xi$ is the $\xi$-th coordinate of $V$.
For each $ U \in \mathcal U$, define $\{R_{U,V}:\, V \in \mathcal U\}$ pairwise disjoint subsets of $R_U \setminus \beta$ of cardinality $\mathfrak c$.

Let $\{ A_\alpha:\, \alpha < \mathfrak c\}$ be an enumeration of $p_\beta$. 

For each $\alpha < \mathfrak c$, define a homomorphism 
$\psi_\alpha:\, \mathbf W\to\mathbb T$ such that $\psi_\alpha$ preserves the assigned accumulation points by the frame, $\psi_\alpha|_{\mathbf W_\omega}=0$,
$\psi_\alpha (\chi_{a_n})=0$ if $n \notin A_\alpha$, $\psi_\alpha(\chi_{a_n})=\frac{1}{2}+\mathbb Z$
if $n \in A_\alpha$ and $\psi_\alpha(\chi_{\mu})=z_V(\alpha)$ whenever $U, V \in \mathcal U$, $\alpha \in \supp V$
and $\mu \in R_{U,V}$. This is possible since $\omega \cup A \cup \bigcup \{ R_{U,V}:\, U, V \in \mathcal U\} \subseteq L_1 \setminus J_1$ is a disjoint union; note that $A\subset\beta$ since by the definition of frame, $\supp h_\beta(n)\subset(\beta\times\omega)\cup\beta$, so $\supp\chi_{a_n}=\{a_n\}\subset\beta$.

Note that since $p_\beta$ is in the fulfillment of the frame and $A_\alpha \in p_\beta$, it follows that $\frac{1}{2}+\mathbb Z=p_\beta$-lim $\{\psi_\alpha(\chi_{a_n}):\, n \in \omega\}=p_\beta$-lim$ \{\psi_\alpha(h_\beta(n)):\, n \in \omega\}=\psi_\alpha(\chi_\beta)$ and 
$A_\alpha = \{n \in \omega:\, \psi_\alpha(h_\beta(n)) \in O\}$, where $O$ is the arc of length $\frac{1}{2}$ centered in $\phi_\alpha(\chi_\beta)$.

Let $\Psi$ be the diagonal of the $\psi_\alpha$'s.

Then $\Delta(\Phi,\Psi):\mathbf W\to \mathbb T^{\mathfrak c}\times \mathbb T^{\mathfrak c}$ is as required since $U\times V$ for $U, V \in \mathcal U$ form a basis 
of basic open sets of $\mathbb T^\mathfrak c \times \mathbb T^\mathfrak c$. 
Note that the convergent sequences are still converging to $0$, and
$p_\beta=\mathcal F(\langle G, \tau_{\Delta(\Phi, \Psi)}\rangle,\chi[A], \chi_\beta) \in \mathcal F(G,\tau_{\Delta(\Phi,\Psi)})$ and $p_\beta=\mathcal F(\langle G\cap \mathbf W_{\mathfrak c\setminus \omega}, \tau_{\Delta(\Phi, \Psi)}|_{ G\cap \mathbf W_{\mathfrak c\setminus \omega}}, \chi[A], \chi_\beta) \in \mathcal F(G\cap \mathbf W_{\mathfrak c\setminus \omega},\tau_{\Delta(\Phi,\Psi)}|_{G\cap \mathbf W_{\mathfrak c\setminus \omega}})$.\\
We can now re-enumerate the indices using a bijection between $\mathfrak c $ and $\mathfrak c \times \mathfrak c$ so that $G$ is immersed in 
$\mathbb T^\mathfrak c$ as required. The basic open set will have only its support renamed, thus the condition of density will be preserved. A basic open $W$ in $\mathbb T^\mathfrak c$ will correspond to a basic open set $U\times V \in \mathbb T^\mathfrak c \times \mathbb T^\mathfrak c$  and $S_W$ will be $R_{U,V}$.
\end{proof}

\begin{lem}\label{p_na_familia_filtro} Assume that there exist $\mathfrak c$ selective ultrafilters. Let $p$ be a selective ultrafilter and
$G$ be a non-torsion Abelian group of cardinality $\mathfrak c$ that admits a countably compact group topology. Then there exists a group topology $\tau$ on  $G$ which makes it a countably compact group with (or without, if desired) converging sequences such that $p\in \mathcal F(G, \tau)$. Moreover, we can choose the weight of $\tau$ to be any cardinal in $[\mathfrak c, 2^\mathfrak c]$.
\end{lem}

\begin{proof}
Consider the frame in Lemma \ref{filter} and set an enumeration $\mathcal P$ for the selective ultrafilters so that $p=p_\beta$. By Theorem \ref{ccgwithcs},
there are $\Phi$ and $(R_U:\, U \in \mathcal U)$ such that the conditions of
Lemma \ref{filter} are satisfied.

Then there exists another fulfillment $\eta$ as in the conclusion of Lemma \ref{filter} in which $p \in \mathcal F (\langle G, \tau_\eta\rangle)$. The fulfillment $\mathcal P$ and $\eta$ satisfies the conditions in
Example \ref{largeweight}, thus, there is a group topology $\tau$ with weight $\kappa$ for which $p \in \mathcal F(\langle G, \tau \rangle )$.

Now we can again repeat the argument for $G\oplus \mathbf {W}_\omega$ used in Corollary \ref{equiv} to obtain the topology without non-trivial convergent sequences for $G$ as desired, as the witness to obtain $p$ depends only on elements of $L_1 \setminus \omega$.
\end{proof}

\begin{cor} Assume the existence of $2^{\mathfrak c}$ selective ultrafilters. For each cardinal $\kappa \in [\mathfrak c , 2^\mathfrak c]$ and each  $G$ that admits a countably compact group topology

\begin{enumerate}
    \item there exists at least $2^\mathfrak c$ non-homeomorphic countably compact group topologies with non-trivial convergent sequences on $G$ of weight $\kappa$ and
    
     \item there exists at least $2^\mathfrak c$ non-homeomorphic countably compact group topologies without non-trivial convergent sequences on $G$ of weight $\kappa$.
    
\end{enumerate}
\end{cor}

\begin{proof}
Fix $\kappa < 2^{\mathfrak{c}}$ a cardinal and consider $\{(X_{\alpha}, \tau_{\alpha}) : \alpha < \kappa\}$ a family of spaces such that $|X_{\alpha}| = \mathfrak{c}$, for every $\alpha < \kappa$. We shall show that it is possible to endow $G$ with a countably compact group topology $\tau$ such that $(G, \tau)$ is not homeomorphic to $(X_{\alpha}, \tau_{\alpha})$, for every $\alpha < \kappa$. Since we are assuming the existence of $2^{\mathfrak{c}}$ selective ultrafilters, let $p$ be a selective ultrafilter such that $p \not \in \bigcup_{\alpha < \kappa} \mathcal{F}(X_{\alpha})$.

By Lemma \ref{p_na_familia_filtro}, we can obtain a group topology $\tau^i$ for $i\in \{1,2\}$ as in items $(1)$ or $(2)$ for 
which $p\in \mathcal F(\langle G, \tau^i\rangle )$. Then this topology is non-homeomorphic to the ones listed. Thus, there must be at least $2^\mathfrak c$ non-homeomorphic topologies as in $(1)$ and $(2)$.
\end{proof}

\section{Question} \label{question}

Recently, Bellini, Rodrigues and Tomita \cite{BELLINI2021107653} showed that if $p$ is  a selective ultrafilter and $\kappa=\kappa^\omega$ then $\mathbb Q^{(\kappa)}$ admits a $p$-compact group topology without non-trivial convergent sequences from a selective ultrafilter. 
This gives the first arbitrarily large non-torsion countably compact groups without non-trivial convergent sequences.

\begin{quest} Classify the non-torsion or the torsion-free Abelian groups of cardinality $\mathfrak c$ that admit a $p$-compact topology (without non-trivial convergent sequences), for some ultrafilter $p$.
\end{quest}

The torsion counterpart of the question above has a consistent classification \cite{castro-pereira&tomita2010}.

In particular, we ask:

\begin{quest} Is there a $p$-compact group topology (without non-trivial convergent sequences) compatible with ${\mathbb Z}^{(\mathfrak c)} \times \mathbb Q ^{(\mathfrak c)} $, for some ultrafilter $p$?  A group topology whose $\omega$-th power is countably compact? What about $\mathbb Z \times \mathbb Q^\mathfrak c$?
\end{quest}

In \cite{boero&castro&tomita}, Boero, 
Castro-Pereira and Tomita showed that there exists a countably compact free Abelian group of cardinality $\mathfrak c$ from a selective ultrafilter.

\begin{quest} Assume the existence of a selective ultrafilter. Classify the Abelian groups of cardinality $\mathfrak c$ that admit a countably compact group topology.
\end{quest}

Recently Hru\v sak, van Mill, Ramos-Garcia and Shelah \cite{michaelnew}
provided a countably compact group of order $2$ without non-trivial convergent sequences in $ZFC$, answering classic problems of van Douwen and Comfort. It is not clear yet if this can be extended to non-torsion groups of cardinality $\mathfrak c$. The following has been asked by Tkachenko and it would be the first step towards a $ZFC$ classification of Abelian groups that admit a countably compact group topology.

\begin{quest}[\cite{van1990}, Question 508] Is there a countably compact group topology on the free Abelian group of cardinality $\mathfrak c$ in $ZFC$?
\end{quest}

\section{Acknowledgments} The first and third authors are doctoral students (Processes FAPESP numbers 2017/15709-6, 2017/15502-2, 2019/01388-9 and 2019/02663-3) who conducted a complete revision and rewriting on the manuscript to its present form under the supervision of the last author. The second author has received financial support from CNPq - "Bolsa de doutorado" and the first draft (the classification without convergent sequences) is part of her doctoral thesis. The last author received financial support from CNPq (Brazil) --- ``Bolsa de Produtividade em Pesquisa (processo 305612/2010-7). Projeto: Grupos topol\'ogicos, sele\c{c}\~oes e topologias de hiperespa\c{c}o" during the research that led to the first draft and has received support from FAPESP Aux\' \i lio regular de pesquisa Proc. Num. 2012/01490-9, CNPq Produtividade em Pesquisa 307130/2013-4, CNPq Projeto Universal Num. 483734/2013-6 and Aux\' \i lio Regular FAPESP - 2016/26216-8 during the research that led to the final version. 

\bibliographystyle{plain}
\bibliography{gruposenumeravelmentecompactos}
\end{document}